\documentclass[11pt]{amsart}
\usepackage{amscd,amsxtra,amssymb, bbm}
\usepackage{tikz}
\usepackage[utf8]{inputenc}

\makeatletter
\newcommand*{\rom}[1]{\expandafter\@slowromancap\romannumeral #1@}
\makeatother

\usepackage{multirow} 
\usepackage{stmaryrd} 

\input xy
\xyoption{all}
\xyoption{arc}
\SelectTips{cm}{10}

\usepackage{geometry}
    \newenvironment{dedication}
        {\begin{quotation}\begin{center}\begin{em}}
        {\par\end{em}\end{center}\end{quotation}}

\newtheorem{theorem}{Theorem}[section]
\newtheorem{corollary}[theorem]{Corollary}
\newtheorem{lemma}[theorem]{Lemma}
\newtheorem{proposition}[theorem]{Proposition}

\theoremstyle{definition}
\newtheorem{definition}[theorem]{Definition}
\newtheorem{remark}[theorem]{Remark}

\newtheorem{example}[theorem]{Example}

\DeclareMathOperator{\rk}{\mathrm{rk}}

\newcommand{\tr}{\mathsf{tr}}

\DeclareMathOperator{\res}{\mathsf{res}}
\DeclareMathOperator{\ev}{\mathsf{ev}}
\DeclareMathOperator{\Sol}{\mathsf{Sol}}

\DeclareMathOperator{\VB}{\mathsf{VB}}

\DeclareMathOperator{\Pic}{\mathsf{Pic}}

\DeclareMathOperator{\Hom}{\mathsf{Hom}}

\DeclareMathOperator{\Ext}{\mathsf{Ext}}

\DeclareMathOperator{\GL}{\mathsf{GL}}
\DeclareMathOperator{\Aut}{\mathsf{Aut}}
\DeclareMathOperator{\End}{\mathsf{End}}
\DeclareMathOperator{\Mat}{\mathsf{Mat}}

\DeclareMathOperator{\lieg}{\mathfrak{g}}
\DeclareMathOperator{\lieh}{\mathfrak{h}}
\DeclareMathOperator{\lien}{\mathfrak{n}}

\DeclareMathOperator{\liew}{\mathfrak{w}}

\DeclareMathOperator{\lieb}{\mathfrak{b}}
\DeclareMathOperator{\liea}{\mathfrak{a}}

\DeclareMathOperator{\lieD}{\mathfrak{D}}
\DeclareMathOperator{\lieP}{\mathfrak{P}}
\DeclareMathOperator{\lieW}{\mathfrak{W}}
\DeclareMathOperator{\lieO}{\mathfrak{O}}

\newcommand{\kk}{\mathbbm{k}}
\newcommand{\CC}{\mathbb{C}}
\newcommand{\NN}{\mathbb{N}}
\newcommand{\ZZ}{\mathbb{Z}}

\newcommand{\PP}{\mathbb{P}}

\newcommand{\kA}{\mathcal{A}}

\newcommand{\kF}{\mathcal{F}}

\newcommand{\kO}{\mathcal{O}}
\newcommand{\kL}{\mathcal{L}}
\newcommand{\kP}{\mathcal{P}}
\newcommand{\kQ}{\mathcal{Q}}

\newcommand{\lar}{\longrightarrow}

\newcommand{\llbrace}{(\!(}
\newcommand{\rrbrace}{)\!)}

\begin{document}

\title[Vector bundles and quasi--trigonometric solutions of CYBE]{Simple vector bundles on a nodal Weierstra\ss{} cubic and quasi--trigonometric solutions of CYBE}

\author{Igor Burban}
\address{
Mathematisches Institut,
Universit\"at zu K\"oln,
Weyertal 86--90,
D--50931 K\"oln,
Germany
}
\email{burban@math.uni-koeln.de}

\author{Lennart Galinat}
\address{Mathematisches Institut,
Universit\"at zu K\"oln,
Weyertal 86--90,
D--50931 K\"oln,
Germany
}
\email{lgalinat@math.uni-koeln.de}

\author{Alexander Stolin}
\address{Department of Mathematical Sciences, Chalmers University of Technology
and the University of Gothenburg, 412 96 Gothenburg, Sweden
}
\email{astolin@chalmers.se}

\begin{abstract}
In this paper we study the combinatorics of quasi--trigonometric solutions of the classical Yang--Baxter equation, arising from simple vector bundles on a nodal Weierstra\ss{} cubic.
\end{abstract}

\maketitle
\begin{dedication}
\vspace*{1mm}{To the memory of Petr Kulish}
\end{dedication}
\section{Introduction}

\noindent
Let $\lieg = \mathfrak{sl}_n(\CC)$ for some $n \in \NN_{\ge 2}$ and $\bigl(\CC^2, 0\bigr) \stackrel{r}\lar \lieg \otimes \lieg$  be the germ of
a meromorphic function.
In this article, we study an interplay between  algebro--geometric and combinatorial aspects  of the theory of  the classical Yang--Baxter equation (CYBE)
\begin{equation}\label{E:CYBE}
\bigl[r^{12}(x_1, x_2), r^{13}(x_1, x_3)\bigr] + \bigl[r^{13}(x_1,x_3), r^{23}(x_2,x_3)\bigr] +
\bigl[r^{12}(x_1, x_2),
r^{23}(x_2, x_3)\bigr] = 0.
\end{equation}
The upper indices in this equation  encode corresponding embeddings of
$\mathfrak{g}\otimes \mathfrak{g}$ into $\mathfrak{a} \otimes \mathfrak{a} \otimes \mathfrak{a}$, where $\liea = \mathfrak{gl}_n(\CC)$.  For example, the function $r^{23}$ is defined as
$$
r^{23}: \bigl(\mathbb{C}^2, 0\bigr) \stackrel{r}\lar \mathfrak{g}\otimes \mathfrak{g}
\stackrel{\imath_{23}}\lar \mathfrak{a} \otimes \mathfrak{a} \otimes \mathfrak{a},
$$
where $\imath_{23}(a\otimes b) = 1 \otimes a  \otimes b$; the other
maps $r^{12}$ and  $r^{13}$  have a similar meaning. A solution $r$ of (\ref{E:CYBE}) is called skew--symmetric if
$r^{12}(x_1, x_2) = - r^{21}(x_2, x_1)$ and non--degenerate if the tensor $r(x_1^\circ, x_2^\circ)$ is non--degenerate for some (hence, for any generic) point $(x_1^\circ, x_2^\circ)$ from the definition domain of $r$. For any germ automorphism $(\CC, 0) \stackrel{\pi}\lar (\CC, 0)$ (change of variables) and a holomorphic map
$(\CC, 0) \stackrel{\phi}\lar \Aut(\lieg)$ (gauge transformation), we get an \emph{equivalent} solution of CYBE:
\begin{equation}\label{E:equivalence}
\widetilde{r}(x_1, x_2):= \bigl(\phi(\pi(x_1)) \otimes \phi(\pi(x_2))\bigr) r\bigl(\pi(x_1), \pi(x_2)\bigr).
\end{equation}
 Clearly, $\widetilde{r}(x_1, x_2)$ is skew--symmetric (respectively, non--degenerate) provided $r(x_1, x_2)$ is skew--symmetric (respectively, non--degenerate). Solutions of CYBE  arise in various areas of algebra, representation theory and mathematical physics;   see for example  \cite{ChariPressley, FaddeevTakhtajan}.

For the Lie algebra  $\lieg = \mathfrak{sl}_2(\CC)$, all solutions of (\ref{E:CYBE}) are  well--known; see \cite{BelavinDrinfeld, Stolin1}.  Namely, consider the following basis $
h =
\left(
\begin{array}{cc}
1 & 0 \\
0 & -1
\end{array}
\right),$
$
e =
\left(
\begin{array}{cc}
0 & 1 \\
0 & 0
\end{array}
\right)$ and $
f =
\left(
\begin{array}{cc}
0 & 0 \\
1 & 0
\end{array}
\right)
$
of  $\lieg$. Then
there exist precisely six (up to the the above equivalence relation (\ref{E:equivalence})) non--equivalent  non--degenerate skew--symmetric solutions of  CYBE.

\smallskip
\noindent
1.~One elliptic solution
\begin{equation}\label{E:elliptic}
r(x, y) =   \frac{\mathrm{cn}(z)}{\mathrm{sn}(z)} h \otimes h +
\frac{1+ \mathrm{dn}(z)}{\mathrm{sn}(z)}(e \otimes f  +
f \otimes e)  +
\frac{1 -  \mathrm{dn}(z)}
{\mathrm{sn}(z)}(e \otimes e + f \otimes f),
\end{equation}
where $z = y-x$.

\smallskip
\noindent
2.~Two (quasi--)trigonometric solutions:
\begin{equation}\label{E:trigonom1}
r(x, y) = \frac{1}{2} \frac{y+x}{y-x} h \otimes h + \frac{y}{y-x} e \otimes f + \frac{x}{y-x} f \otimes e
\end{equation}
and
\begin{equation}\label{E:trigonom2}
r(x, y) = \frac{1}{2} \frac{y+x}{y-x} h \otimes h + \frac{y}{y-x} e \otimes f + \frac{x}{y-x} f \otimes e + (x-y) f \otimes f.
\end{equation}

\smallskip
\noindent
3.~Three  rational solutions:
\begin{equation}\label{E:rat1}
r(x, y) = \frac{1}{y-x}\left(\frac{1}{2} h \otimes h + e \otimes f + f \otimes e\right),
\end{equation}
\begin{equation}\label{E:rat2}
r(x, y) = \frac{1}{y-x}\left(\frac{1}{2} h \otimes h + e \otimes f + f \otimes e\right) + y h \otimes f - x f \otimes h
\end{equation}
and
\begin{equation}\label{E:rat3}
r(x, y) = \frac{1}{y-x}\left(\frac{1}{2} h \otimes h + e \otimes f + f \otimes e\right) + h \otimes f - f \otimes h.
\end{equation}
Let  $\lieg = \mathfrak{sl}_n(\CC)$ and $a, b \in \lieg$ be such that $[a, b] = 0$. It is not difficult to check that
$$
r(x, y) = \frac{\gamma}{y-x}  + a \wedge b
$$
is a solution of (\ref{E:CYBE}), where $\gamma \in \lieg \otimes \lieg$ is the Casimir element and $a \wedge b = a \otimes b - b \otimes a$. As a consequence,
the  description  of \emph{all} solutions of (\ref{E:CYBE}) contains as a subproblem the problem of classification of all pairs of commuting square matrices. The latter problem is  known to be \emph{representation--wild}; see e.g.~\cite{GelfandPonomarev}.

In  \cite{KPSST}, the following class of the so--called \emph{quasi--trigonometric} solutions of (\ref{E:CYBE}) was introduced.  These solutions   are of the form
\begin{equation}\label{E:quasitrigon}
r(x, y) = \frac{x}{y-x} \gamma + p(x, y) = \left(\sum\limits_{n = 1}^\infty x^n y^{-n}\right) \gamma + p(x,y),
\end{equation}
where $\gamma \in \lieg \otimes \lieg$ is the Casimir element and  $p(x, y) \in (\lieg \otimes \lieg)[x, y]$.
The simplest example is the so--called \emph{standard} quasi--trigonometric solution, given by the following expression:
\begin{equation}\label{E:StandardQuasiTrig}
r_{\mathsf{st}}(x, y) = \dfrac{1}{2}\Bigl(\dfrac{y+x}{y-x} \gamma + \sum\limits_{\alpha \in \Phi_+} e_{\alpha} \wedge e_{-\alpha}\Bigr),
\end{equation}
where $\Phi_+$ denotes the set of all  positive roots of $\lieg$. For example, for  $\lieg = \mathfrak{sl}_2(\CC)$, the solution $r_{\mathsf{st}}$ is given by the formula (\ref{E:trigonom1}).

Quasi--trigonometric solutions of (\ref{E:CYBE}) have a number of special properties. Firstly, any quasi--trigonometric solution $r$
 is  automatically skew--symmetric: $r^{12}(x, y) = - r^{21}(y, x)$ (see \cite[Proposition 5]{KPSST}) and
equivalent to a \emph{trigonometric} solution (hence the name) in the sense of the Belavin--Drinfeld classification \cite{BelavinDrinfeld}
(see \cite[Theorem 19]{KPSST}). Secondly,  $r$  defines a Lie cobracket
$\delta_r$
 on the Lie algebra $\lieg[z]:= \lieg \otimes_{\CC} \CC[z]$, given by the rule
 $$
 \lieg[z] \stackrel{\delta_r}\lar \lieg[z_1] \otimes \lieg[z_2], \quad P(z) \mapsto \bigl[P(z_1) \otimes 1 + 1 \otimes P(z_2), r(z_1, z_2) \bigr].
 $$
 Hence, quasi--trigonometric solutions of CYBE  play  an important role in the classification of \emph{all} bialgebra structures on $\lieg[z]$; see \cite{MSZ}.

In this paper, we study certain  \emph{distinguished} quasi--trigonometric solutions of CYBE for the Lie algebra $\lieg = \mathfrak{sl}_n(\CC)$, attached to any natural number
$1 \le c < n$ such that
$\mathrm{gcd}(n, c) = 1$.  Let
$$\bar{\Phi} := \bigl\{(i, j) \in \NN^2 \,\big|\, 1 \le i, j \le n \bigr\} \cong \ZZ/n\ZZ \times \ZZ/n\ZZ$$
and
$\Phi_+ := \bigl\{(i, j) \in \bar{\Phi} \,\big|\,  i < j \bigr\}$. Of course, one can identify $\Phi_+$ with the set of all positive  roots
 of the Lie algebra $\lieg$. Then we have a permutation $\tau_c$ acting on the set $\bar{\Phi}$ by the following rule:
$$
\bar{\Phi} \stackrel{\tau_c}\lar \bar{\Phi}, \; (i, j) \mapsto (i+c, j+c).
$$
Since $\gcd(n, c) = 1$, the order of  $\tau_c$ is  $n$.  For any positive root $\alpha \in \Phi_+$, let
$
p_c(\alpha) = \min\left\{k \in \NN \, \big|\, \tau_c^k(\alpha) \notin \Phi_+\right\}.
$
For any $1 \le i \le n-1$, we put:
$
q_{i, c} := \tau_c^{i}(u) - \tau_c^{i-1}(u)$ and $f_{i, c} := \frac{1}{2}\bigl(\tau_c^{i}(u) + \tau_c^{i-1}(u)\bigr) - \frac{1}{n} I,
$
where $I$ is the identity matrix and $u = e_{11}$ is the first matrix unit. Then $(q_{1, c}, \dots, q_{n-1, c})$ is a basis of the standard Cartan part $\lieh$ of the Lie
algebra $\lieg$. Let $(q_{1, c}^\ast, \dots, q_{n-1, c}^\ast)$ be the dual basis of  $\lieh$ with respect to the trace form.
Then  the following result is true; see Theorem \ref{T:maincomputation} and Corollary \ref{C:Limits}.

\smallskip
\noindent
\textbf{Theorem A}. For any $1 \le c < n$ such that $\mathrm{gcd}(n, c) =1$, consider the meromorphic function
$\CC^2 \stackrel{r_c}\lar \lieg \otimes \lieg$, given by the formula
$
r_{c}(x, y) = r_{\mathsf{st}}(x, y) + u_c(x, y) + t_c,
$
where
$$
u_{c}(x,y) = \sum\limits_{\alpha \in \Phi_+}\Bigl(\Bigl(\sum\limits_{k = 1}^{p_c(\alpha)-1} e_{\tau_c^k(\alpha)} \wedge e_{-\alpha}\Bigr) +
x e_{\tau_c^{p_c(\alpha)}(\alpha)} \otimes
e_{-\alpha}  - y e_{-\alpha} \otimes x e_{\tau_c^{p_c(\alpha)}(\alpha)}\Bigr) \in (\lieg \otimes \lieg)[x, y]
$$
and $t_{c} = \sum_{i=1}^{n-1} q_{i, c}^\ast \otimes f_{i, c} \in \lieh \otimes \lieh$.
Then $r_{c}$ is a (quasi--trigonometric) solution of the classical Yang--Baxter equation (\ref{E:CYBE}). Moreover, the solutions $r_c$ and $r_d$ are gauge--equivalent, where
$d = n-c$. Next, $r_c$ is equivalent  to a degeneration  of Belavin's elliptic solution \cite{Belavin} corresponding to the primitive $n$--root of unity $\varepsilon = \exp\left(\dfrac{2\pi i c}{n}\right)$. Finally,  $r_c$ itself degenerates (in  an appropriate sense) to a \emph{distinguished rational solution} given by the formula from
\cite[Theorem 9.6]{BH}.

\smallskip
\noindent
\textbf{Example}.
For $n = 2$ and $c = 1$, the solution $r_c$ is given by the formula (\ref{E:trigonom2}). As it was already pointed out in \cite[Chapter 6]{BK4}, the quasi--trigonometric solution (\ref{E:trigonom2}) is equivalent to an appropriate  degeneration of the elliptic solution (\ref{E:elliptic}) and can be degenerated into the rational solution (\ref{E:rat2}). From the analytical point of view, a proof of the latter fact is not straightforward.

\smallskip
\noindent
For $n =3$,  the formula for $r_c$ takes the following explicit form.
\begin{itemize}
\item If $c = 1$ then
$$
u_c(x, y) = e_{23} \wedge e_{21} + (x-y)(e_{31} \otimes e_{21} + e_{21} \otimes e_{31}) + (x e_{31} \otimes e_{32} - y e_{32} \otimes e_{31})
$$
and $t_c = l_2 \wedge l_1$, where $l_1 = \mathsf{diag}(\frac{1}{3}, \frac{1}{3}, -\frac{2}{3})$ and $l_2 = \mathsf{diag}(\frac{2}{3}, -\frac{1}{3}, -\frac{1}{3})$.
\item If $c = 2$ then
$$
u_c(x, y) = e_{12} \wedge e_{32} + (x-y)(e_{31} \otimes e_{32} + e_{32} \otimes e_{31}) + (x e_{31} \otimes e_{21} - y e_{21} \otimes e_{31})
$$
and $t_c = l'_2 \wedge l'_1$, where $l'_1 = \mathsf{diag}(\frac{1}{3},  -\frac{2}{3}, \frac{1}{3})$ and $l'_2 = \mathsf{diag}(-\frac{2}{3}, \frac{1}{3}, \frac{1}{3})$.
\end{itemize}

\smallskip
\noindent
 According to the theory developed in \cite{KPSST},  the quasi--trigonometric solutions of (\ref{E:CYBE})
 can be classified (up to an appropriate equivalence relation) in the following terms. Let $\lieD := \lieg\llbrace z^{-1}\rrbrace \times \lieg$.
 Then we have a symmetric non--degenerate $\CC$--bilinear form
\begin{equation}
\langle-\; ,\,-\,\rangle:  \lieD \times \lieD \lar \CC, \quad \bigl\langle(F, f), (G, g)\bigr\rangle = \res_0 \left(\mathsf{Tr}(FG) \frac{dz}{z}\right) - \tr(f g),
\end{equation}
where $\liea\llbrace z^{-1}\rrbrace \stackrel{\mathsf{Tr}}\lar \CC\llbrace z^{-1}\rrbrace$ (respectively, $\liea \stackrel{\mathsf{tr}}\lar \CC$) is the trace map.
Note that  we have  an embedding
$\lieg[z] \stackrel{\jmath}\lar \lieD, \; P(z) \mapsto \bigl(P(z), P(0)\bigr)$
identifying   $\lieg[z]$ with a \emph{Lagrangian} Lie subalgebra $\lieP$ of $\lieD$.

 As it was shown in \cite{KPSST}, the quasi--trigonometric solutions of (\ref{E:CYBE}) are parameterized by Lagrangian subalgebras $\lieW \subset \lieD$ transversal to $\lieP$, for which there exists
some $m\in \NN$ (depending on $\lieW$)  such that $z^{-m} \lieg\llbracket z^{-1}\rrbracket \times \{0\}\subseteq \lieW$.
In other words, we have a \emph{Manin triple} in the Lie algebra $\lieD$ of the following form:
\begin{equation}
\lieD = \lieP \dotplus \lieW.
\end{equation}

On the other hand,   solutions of (\ref{E:CYBE}) can be studied using methods of algebraic geometry; see  \cite{Cherednik, Polishchuk1, Polishchuk2, BK4, BH, BurbanGalinat}.
Namely, let $
E = \overline{V\bigl(u^2 - 4v^3 + g_2 v + g_3\bigr)} \subset \PP^2
$
 be a Weierstra\ss{} cubic curve for some parameters  $g_2, g_3 \in \CC$. Such a curve is singular (nodal or cuspidal) if and only if $g_2^3  =  27 g_3^2$; in this case $E$ has a unique singular point $s$. Assume that   $\kA$ is a locally free coherent sheaf of Lie algebras on $E$ such that:
\begin{equation}\label{E:GeomAnsatz}
H^0(E, \kA) = 0 = H^1(E, \kA)\;\;  \mbox{\rm and}\;\;   \kA\big|_x \cong \lieg \;\;  \mbox{\rm for any smooth point}\;\; x \in E.
\end{equation}
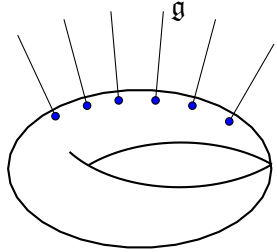
\begin{center}
\begin{tikzpicture}[samples=35, scale = 0.7, dot/.style={draw,circle,minimum size=1mm,inner sep=0pt,outer sep=0pt,fill=blue}]
  \def\a{2.5}
  \def\b{1.5}
  \def\PI{3.14159265359}
  \draw[thick, domain=0:2*\PI] plot ({\a*cos(\x r)},{\b*sin(\x r)});
  \draw[thick, domain=\PI/4+0.02 : -0.01+3*\PI/4]  plot ({\a*cos(\x r)+0.78},{\b*sin(\x r) -1});
  \draw[thick, domain=-0.2+5*\PI/4: -0.01+7*\PI/4] plot ({\a*cos(\x r)+0.75},{\b*sin(\x r) +1.15});

   \draw (1.7, 0.9)  node[dot]{}  -- ++(60: 1.7cm);
   \draw (1, 1.2)  node[dot]{}  -- ++(75: 1.7cm);
   \draw (0.3,	1.3)  node[dot]{}  -- ++( 85: 1.7cm) node[ xshift= 0.2cm] {$\lieg$};
   \draw (-0.4, 1.3)  node[dot]{}  -- ++(95: 1.7cm);
   \draw (-1, 1.2)  node[dot]{}  -- ++(105: 1.7cm);
   \draw (-1.6,1)  node[dot]{}  -- ++(115: 1.7cm);
\end{tikzpicture}
\end{center}
Under these assumptions on $\kA$ (which can be weakened, when one replaces the constraint $\kA\big|_{s} \cong \lieg$ at the singular point $s \in E$ by a more general condition \cite{BurbanGalinat}), there exists
a distinguished section
$
\rho \in \Gamma\bigl(\breve{E} \times \breve{E}\setminus \Sigma, \kA \otimes \kA\bigr)$ (called \emph{geometric $r$--matrix}),
where $\breve{E}$ is the regular part of $E$ and $\Sigma \subset \breve{E} \times \breve{E}$ is the diagonal. It turns out that $\rho$   satisfies the following version of the classical Yang--Baxter equation: $$
\bigl[\rho^{12}, \rho^{13}\bigr] + \bigl[\rho^{12}, \rho^{23}\bigr] + \bigl[\rho^{13}, \rho^{23}\bigr] = 0,
$$
where both sides of the above  equality are viewed as meromorphic sections of   $\kA \boxtimes \kA \boxtimes \kA$ over the triple product $\breve{E} \times \breve{E} \times \breve{E}$.
Moreover, the section $\rho$ is skew--symmetric, i.e.
$$
\rho(x_1, x_2)^{12} = - \rho(x_2, x_1)^{21} \in \kA\big|_{x_1} \otimes \kA\big|_{x_2} \quad \mbox{\rm for any} \; x_1 \ne x_2 \in \breve{E}.
$$
In order to get a link with the conventional form of CYBE (\ref{E:CYBE}), suppose that there exists an open subset
$C \subset \breve{E}$ and  a $\Gamma(C, \kO_E)$--linear isomorphism of Lie algebras $$\Gamma(C, \kA) \stackrel{\xi}\lar \lieg \otimes_\CC \Gamma(C, \kO_E).$$
 This trivialization $\xi$ allows to rewrite  the geometric $r$--matrix $\rho$  as  a meromorphic  function
 $$r = \rho^\xi: C \times C\setminus\Sigma\lar \lieg \otimes \lieg,$$ which is a non--degenerate skew--symmetric solution of the  classical Yang--Baxter equation (\ref{E:CYBE}).
 Another choice  of a local trivialization of $\kA$ leads to a \emph{gauge--equivalent} solution.

A natural class of sheaves of Lie algebras $\kA$ satisfying the condition (\ref{E:GeomAnsatz}) arises from the following construction. Let
 $\kP$ be a \emph{simple} vector bundle on a Weierstra\ss{} curve $E$ (i.e.~$\End_E(\kP) = \CC$), $n = \rk(\kP)$ be its rank and $d = \deg(\kP)$ be its degree.
 The theory of simple vector bundles on Weierstra\ss{} cubics is well--understood \cite{Atiyah, BodnarchukDrozd, Burban1}.
 It turns out that $\gcd(n, d) = 1$ and for any other simple vector bundle $\kQ$ of rank $n$ and degree $d$ on $E$, there exists a line bundle
 $\kL \in \Pic^0(E)$ such that $\kQ \cong \kP \otimes \kL$. Moreover, for any $(n,d) \in \NN \times \ZZ$ satisfying the condition $\gcd(n, d) = 1$, there exists a simple vector bundle of rank $n$ and degree $d$ on $E$.

 \smallskip
 \noindent
 Let $\kA  = \mbox{\it Ad}_E(\kP)$ be the sheaf of Lie algebras on $E$ given by the short exact sequence
\begin{equation}\label{E:DefAd}
0 \lar \kA \lar \mbox{\it End}_E(\kP) \stackrel{\mathsf{tr}}\lar \kO \lar 0.
\end{equation}
From what was said above it follows that  $\kA = \kA_{(n,d)}$ does not depend on the particular choice of simple vector bundle $\kP$ (and as a consequence,  is uniquely determined
by the pair $(n, d)$ such that $\gcd(n,d) = 1$) and satisfies the constraints (\ref{E:GeomAnsatz}). Moreover, it follows that $\kA_{(n, d)} \cong \kA_{(n, n+d)}$.  Summing up, the geometric $r$--matrix attached
to the pair $(E, \kA)$, defines a non--degenerate skew--symmetric solution $r_{(E, (n,d))}$ of (\ref{E:CYBE}) for the Lie algebra $\lieg = \mathfrak{sl}_n(\CC)$,
whose type is fully determined by the type of the underlying Weierstra\ss{} curve $E$ and a natural number $0 < d < n$, which is mutually prime to $n$.

It is therefore a very natural problem to determine the corresponding solutions of (\ref{E:CYBE}) explicitly. It turns out, that for an elliptic curve
$E$, one gets precisely the elliptic solutions of Belavin \cite{Belavin}, where a choice of $0 <d < n$ mutually prime to $n$
corresponds precisely to a choice of a primitive $n$-th root of $1$; see  \cite[Theorem 5.5]{BH}. For a cuspidal curve $E$,
one gets  a distinguished rational solution of (\ref{E:CYBE}), whose combinatorics
(in the sense of the works of the third--named author \cite{Stolin1, Stolin2}) was determined in \cite[Theorem 9.8]{BH}. In this paper, we treat  the nodal case.

Let $0 < d < n$ be such that $\gcd(n, d) = 1$. For any $X \in \lieg$, we use the following notation:
\begin{equation}\label{E:partition}
X = \left(
\begin{array}{c|c}
A & B \\
\hline
C & D
\end{array}
\right)
\quad \mbox{\rm and} \quad
X^\sharp := \left(
\begin{array}{c||c}
D & C \\
\hline\hline
B & A
\end{array}
\right),
\end{equation}
where $A$ and $D$ are square matrices of sizes $c\times c$ (for $c := n-d$) and $d \times d$ respectively. In this notation, we put:
\begin{equation}\label{E:solieb}
\lien_{(c,d)} :=
\left\{
\left(
\begin{array}{c|c}
0 & 0 \\
\hline
C & 0
\end{array}
\right)
\right\}
\subset
\lieb_{(c,d)} :=
\left\{
\left(
\begin{array}{c|c}
A & 0 \\
\hline
C & D
\end{array}
\right)
\right\} \subset \lieg.
\end{equation}
The main result of this article is the following (see Theorem \ref{T:CremmerGervaisOrder}, Theorem \ref{T:maincomputation} and Theorem \ref{T:GeomRmatrExplicit}).

\smallskip
\noindent
\textbf{Theorem B}. Let $0 < d < n$ be such that $\gcd(n,d) = 1$.  Then the vector  space $$\lieW_{(c,d)} := \bigl(z^{-2} \lieg\llbracket z^{-1}\rrbracket + z^{-1} \lieb_{(c,d)} + \lien_{(c,d)}\bigr) \times \{0\} + \widetilde{\Delta}_{(c,d)} \subset \lieD,$$
where
\begin{equation}\label{E:twisteddiagonal}
\widetilde{\Delta}_{(c,d)} :=
\left\{
\left(
\left(
\begin{array}{c|c}
A & z^{-1}B \\
\hline
zC & D
\end{array}
\right),
\left(
\begin{array}{c||c}
D & C \\
\hline\hline
B & A
\end{array}
\right)
\right)
\left| \; \mbox{\rm for all} \;  \left(
\begin{array}{c|c}
A & B \\
\hline
C & D
\end{array}
\right) \in \lieg \right.
\right\},
\end{equation}
 is a Lagrangian Lie subalgebra of $\lieD$  such that
$\lieD = \lieP \dotplus \lieW$.  Moreover, the  quasi--tri\-go\-no\-met\-ric solution corresponding to $\lieW_{(c,d)}$, is precisely the solution
 $r_c$ from Theorem A. It is equivalent to the geometric $r$--matrix
$r_{(E, (n,d))}$, where $E$ is a nodal Weierstra\ss{} curve.

\medskip
\noindent
We hope that the geometric study of quasi--trigonometric solutions of CYBE will find applications in  the theory of integrable systems \cite{FaddeevTakhtajan}.

\medskip
\noindent
\emph{Acknowledgement}. The work of the first- and the second--named authors  was supported  by the  project Bu--1866/3--1 as well as by
the CRC/TRR 191 project ``Symplectic Structures in Geometry, Algebra and Dynamics'' of German Research Council (DFG).

\section{Review the geometric theory of the classical Yang--Baxter equation}

\noindent
We begin with a quick  review of the geometric theory of the classical Yang--Baxter equation (\ref{E:CYBE}),
following the exposition of \cite{BurbanGalinat}; see also  \cite{Cherednik, Polishchuk1, Polishchuk2, BK4, BH}. For $g_2, g_3 \in \CC$, let
\begin{equation}\label{E:WeierstrassFamily}
E = \overline{V\bigl(u^2 - 4v^3 + g_2 v + g_3\bigr)} \subset \PP^2
\end{equation}
 be the corresponding  Weierstra\ss{} cubic curve.
It is well--known that
\begin{itemize}
\item $E$ is smooth (i.e.~$E$ is an elliptic curve) if and only if $g_2^3  \ne   27 g_3^2$.
\item If $g_2^3  =   27 g_3^2$ then $E$ has a unique singular point $s$, which is
\begin{itemize}
\item a nodal singularity  if $(g_2, g_3) \ne (0, 0)$,
\item respectively a cuspidal singularity  if $(g_2, g_3) =  (0, 0)$.
\end{itemize}
\item We have: $\Gamma(E, \Omega) \cong \CC$, where $\Omega$ is the sheaf of regular differential one--forms on $E$ (taken in the Rosenlicht sense
if  $E$ is singular; see e.g.~\cite[Section II.6]{BarthHulekPetersVen}).
\end{itemize}
As the next ingredient, we need  a  coherent sheaf of Lie algebras $\kA$ on $E$ such that:
\begin{enumerate}
\item $H^0(E, \kA) = 0 = H^1(E, \kA)$;
\item $\kA$ is weakly $\lieg$--locally free on $\breve{E}$, i.e.~$\kA\big|_x \cong \lieg$ for all $x \in \breve{E}$.
\end{enumerate}
From the first assumption it follows that  the sheaf $\kA$ is torsion free on $E$ (in particular, its restriction $\kA^\circ := \kA\big|_{\breve{E}}$ on the regular part $\breve{E}:= E \setminus \{s\}$ is
 locally free).
The second assumption on $\kA$  implies that
\begin{itemize}
\item The  canonical isomorphism of $\kO_{\breve{E}}$--modules $\kA^\circ  \otimes \kA^\circ \lar \mathit{End}_{\breve{E}}\bigl(\kA^\circ\bigr)$, induced by the Killing forms of the Lie algebras of local sections of $\kA$, is an isomorphism.
\item The space  $A_K$ of global sections of the rational envelope of $\kA$ is  a simple Lie algebra over the field  $K$ of meromorphic functions on $E$.
\end{itemize}
A choice of a global regular one--form $0 \ne \omega \in \Gamma(E, \Omega)$ defines the so--called residue short exact sequence (see for instance \cite[Section 3.1]{BurbanGalinat}):
\begin{equation}\label{E:residue1}
0 \lar \kO_{E \times \breve{E}} \lar \kO_{E \times \breve{E}}(\Sigma) \xrightarrow{\res_\Sigma^\omega} \kO_\Sigma \lar 0,
\end{equation}
where $\Sigma \subset E \times \breve{E}$ denotes the diagonal. Tensoring  (\ref{E:residue1})  with $\kA \boxtimes \kA^\circ$ and applying then the functor $\Gamma(E \times \breve{E}, \,-\,)$, we get
an injective  $\CC$--linear map $$\End_{\breve{E}}(\kA^\circ) \stackrel{T_\omega}\lar \Gamma\bigl(\breve{E} \times \breve{E} \setminus \Sigma, \kA^\circ \boxtimes \kA^\circ\bigr),$$ making the following diagram
\begin{equation*}
\begin{array}{c}
\xymatrix{
\Gamma\bigl(\breve{E}, \kA^\circ \otimes \kA^\circ\bigr) \ar[d]_-\cong & & \Gamma\bigl(E \times \breve{E}, \kA \boxtimes \kA^\circ(\Sigma)\bigr) \ar[ll]_-{\cong} \ar@{^{(}->}[d] \\
 \End_{\breve{E}}(\kA^\circ) \ar[rr]^-{T_\omega} & & \Gamma\bigl(\breve{E} \times \breve{E} \setminus \Sigma, \kA^\circ \boxtimes \kA^\circ\bigr)
}
\end{array}
\end{equation*}
commutative. In this way, we get a \emph{distinguished section}
\begin{equation}
\rho:= T_\omega(\mathrm{id}_{\kA^\circ}) \in \Gamma\bigl(\breve{E} \times \breve{E} \setminus \Sigma, \kA^\circ \boxtimes \kA^\circ\bigr),
\end{equation}
called \emph{geometric $r$--matrix} attached to a pair $(E, \kA)$ as above.

\smallskip
\noindent
If the curve $E$ is singular, we additionally require  that
\begin{enumerate}
  \setcounter{enumi}{2}
  \item
 the germ $A_s$ of the sheaf $\kA$ at the singular point $s$ is a coisotropic Lie subalgebra of $A_K$ with respect to the pairing
  $$
\kappa^\omega: A_K \times A_K \stackrel{\kappa}\lar K \stackrel{\res_s^\omega}\lar \CC,
$$
where $\kappa$ is the Killing form of $A_K$ and $\res_s^\omega(f) = \res_s(f\omega)$ for $f \in K$ (taken in the Rosenlicht sense).
\end{enumerate}

\smallskip
\noindent
Then the following result is true; see  \cite[Theorem 4.3]{BurbanGalinat}.
\begin{theorem}\label{T:GeoemtryCYBE}
Let $(E, \kA)$ be a pair satisfying the properties (1)--(3) above.
Then we have:

\smallskip
\noindent
1.~As the name suggests, the geometric $r$--matrix $\rho$ satisfies the following version of the classical Yang--Baxter equation: $$
\bigl[\rho^{12}, \rho^{13}\bigr] + \bigl[\rho^{12}, \rho^{23}\bigr] + \bigl[\rho^{13}, \rho^{23}\bigr] = 0,
$$
where both sides of the above  equality are viewed as meromorphic sections of   $\kA \boxtimes \kA \boxtimes \kA$ over the triple product $\breve{E} \times \breve{E} \times \breve{E}$.

\smallskip
\noindent
2.~Moreover, the section $\rho$ is skew--symmetric, i.e.
$$
\rho(x_1, x_2)^{12} = - \rho(x_2, x_1)^{21} \in \bigl(\kA \boxtimes \kA\bigr)\big|_{(x_1, x_2)} \cong \kA\big|_{x_1} \otimes \kA\big|_{x_2} \quad \mbox{\rm for any} \; x_1 \ne x_2 \in \breve{E}.
$$

\smallskip
\noindent
3.~Finally, there exists an open subset $U \subseteq \breve{E}$ such that for any $x_1 \ne x_2 \in U$, the tensor $\rho(x_1, x_2) \in \kA\big|_{x_1} \otimes \kA\big|_{x_2}$ is non--degenerate.
\end{theorem}

\begin{remark}
Suppose that $E$ is singular and  $\kA\big|_{s} \cong \lieg$. Then the condition (3) on the stalk $A_s$ is automatically fulfilled.
\end{remark}

\medskip
\noindent
Assume that there exists an open subset
$C \subset \breve{E}$ and  an $R$--linear isomorphism of Lie algebras $\Gamma(C, \kA) \stackrel{\xi}\lar \lieg \otimes_\CC R$, where $R = \Gamma(C, \kO_E)$.
 This trivialization $\xi$ allows to present the geometric $r$--matrix $\rho$  as  a meromorphic  function
$r = \rho^\xi: C \times C\setminus\Sigma\lar \lieg \otimes \lieg,$ which is a non--degenerate solution of the  conventional classical Yang--Baxter equation (\ref{E:CYBE}).

\smallskip
\noindent
A different  choice of an  $R$--linear isomorphism of Lie algebras $\Gamma(C, \kA) \stackrel{\zeta}\lar \lieg \otimes_\kk R$ leads to  an $R$--linear Lie algebra automorphism  $\phi \in \Aut_R\bigl(\lieg \otimes_\CC R\bigr)$, making the following  diagram
    $$
    \xymatrix{
    & \Gamma(C, \kA^\circ) \ar[ld]_-\xi \ar[rd]^-\zeta & \\
 \lieg \otimes_\CC R \ar[rr]^-\phi & &    \lieg \otimes_\CC R
    }
    $$
    commutative.
In these terms we have: $
\bigl(\phi(x_1) \otimes \phi(x_2)\bigr)\rho^\xi(x_1, x_2) = \rho^\zeta(x_1, x_2),
$
for any $x_1 \ne x_2 \in C$, i.e.~the solutions $\rho^\xi$ and $\rho^\zeta$ are \emph{gauge--equivalent}.

Let $r(x_1, x_2)$ be any non--degenerate and skew--symmetric solution of (\ref{E:CYBE}). According to a result
of   Belavin and Drinfeld \cite{BelavinDrinfeld2},  there exists a germ endomorphism $(\CC, 0) \stackrel{\pi}\lar (\CC, 0)$ and a gauge transformation
$(\CC, 0) \stackrel{\phi}\lar \Aut(\lieg)$ such that the meromorphic germ
$$
\widetilde{r}(x_1, x_2):= \bigl(\phi(x_1) \otimes \phi(x_2)\bigr) r\bigl(\pi(x_1), \pi(x_2)\bigr)
$$
depends only on the difference of the spectral variables $x_2-x_1$. In other words, there exists a meromorphic germ
$(\CC, 0) \stackrel{\varrho}\lar \lieg$ such that $\widetilde{r}(x_1, x_2) =  \varrho(x_2-x_1)$. In these terms, the classical Yang--Baxter equation (\ref{E:CYBE}) reduces to the form:
\begin{equation}\label{E:CYBE1}
\bigl[\varrho^{12}(x), \varrho^{13}(x+y)\bigr] + \bigl[\varrho^{13}(x+y), \varrho^{23}(y)\bigr] +
\bigl[\varrho^{12}(x), \varrho^{23}(y)\bigr] = 0
\end{equation}
Another  result of Belavin and Drinfeld \cite{BelavinDrinfeld} asserts that any non--degenerate solution of (\ref{E:CYBE1}) is either elliptic, or
trigonometric or rational. In the geometric terms, the type of the geometric $r$--matrix attached to a pair $(E, \kA)$ is
determined by the type of the Weierstra\ss{} curve $E$: elliptic curves correspond to  elliptic solutions, nodal curves give trigonometric solutions and cuspidal curves lead to rational solutions.
Moreover, at least all elliptic and rational solutions arise from an appropriate geometric $r$--matrix;  see \cite[Remark 4.13 and Theorem 5.3]{BurbanGalinat} as well as  references therein.

\section{Simple vector bundles on a nodal Weierstra\ss{} curve}\label{S:BundlesNodalCurve}

\noindent
From now on, let $E$ be a nodal Weierstra\ss{} curve and $\PP^1 \stackrel{\nu}\lar E$ be its normalization.  We can choose homogeneous coordinates
$(z_0: z_1)$ on $\PP^1$ in such a way that $\nu^{-1}(s) = Z:= \{0, \infty\}$, where $0 := (1:0)$ and $\infty := (0:1)$. The choice of homogeneous coordinates on $\PP^1$ also determines two distinguished sections $z_0, z_1 \in \Gamma\bigl(\PP^1, \kO_{\PP^1}(1)\bigr)$, vanishing at the points $\infty$ and $0$, respectively.
For any $k \in \NN$, we get a distinguished basis
$z_0^k, z_0^{k-1} z_1, \dots, z_1^k$ of the vector space $\Hom_{\PP^1}\bigl(\kO_{\PP^1}, \kO_{\PP^1}(k)\bigr)$. According to a theorem of Birkhoff--Grothendieck, any vector bundle  $\kF$ on $\PP^1$ splits into a direct sum of line bundles:
\begin{equation}\label{E:BirkGrothDecomp}
\kF \cong \bigoplus\limits_{k \in \ZZ} \bigl(\kO_{\PP^1}(k)\bigr)^{\oplus n_k}.
\end{equation}
For any $k \in \NN$, we have
a trivialization $
\kO_{\PP^1}(k)\Big|_{Z} \stackrel{\xi_k}\lar \kO_Z,
$
given on the level of local sections by the rule
\begin{equation}\label{E:trivialization}
\kO_{\PP^1}(k)\Big|_{Z} \stackrel{\xi_k}\lar \kO_Z = \CC_0 \times \CC_\infty, \quad g \mapsto \left(\dfrac{g}{z_0^k}, \dfrac{g}{z_1^k}\right).
\end{equation}
In this way, for any vector bundle $\kF$ with a fixed direct sum decomposition (\ref{E:BirkGrothDecomp}), we get the induced  trivialization $\xi_\kF:
\kF\Big|_{Z} \lar \kO_Z^n,
$
where $n := \sum\limits_{k \in \ZZ} n_k$ is the rank of $\kF$.

\smallskip
\noindent
Let $k, l \in \ZZ$ be such that   $k \le l$ and $\kO_{\PP^1}(k) \stackrel{g}\lar \kO_{\PP^1}(l)$ be a morphism given by a homogeneous polynomial
$g \in \CC[z_0, z_1]$ of degree $l-k$. Then the following diagram is commutative:
$$
\xymatrix{
\kO_{\PP^1}(k)\big|_Z \ar[rr]^{g\big|_Z} \ar[d]_-{\xi_k}& & \kO_{\PP^1}(l)\big|_Z \ar[d]^-{\xi_l}\\
\CC_0 \times \CC_\infty \ar[rr]^{\left(\begin{smallmatrix} g(1,0) & 0 \\ 0 & g(0,1) \end{smallmatrix}\right)} & & \CC_0 \times \CC_\infty
}
$$
In order to describe vector bundles on the nodal Weierstra\ss{} cubic $E$, consider the following Cartesian diagram in the category of schemes:
$$
\xymatrix{
Z \ar@{^{(}->}[r]^-\imath \ar@{->>}[d]_-\jmath & \PP^1 \ar@{->>}[d]^-{\nu} \\
\{s\} \ar@{^{(}->}[r] & E.
}
$$
The key  idea  in the  study of  the category $\VB(E)$ of vector bundles   on $E$ (or, more generally, the category of coherent torsion free sheaves on an arbitrary singular (rational) curve) is to realize it as a full subcategory of the comma category $\mathsf{Comma}(E)$, attached with a pair of functors
$$
\xymatrix{
\VB(\PP^1) \ar[r]^-{\imath^\ast} & \bigl(\CC\times \CC\bigr)-\mathsf{mod} &  \ar[l]_-{\jmath^\ast} \CC-\mathsf{mod}}.
$$
Any object of this comma--category is a triple $\bigl(\kF, \CC^n, \Theta\bigr)$, where $\kF$ is a locally free sheaf on $\PP^1$ and
 $\jmath^*(\CC^n) \stackrel{\Theta}\lar
\imath^*\bigl(\kF\bigr)$ is the gluing map. Using the trivialization $\xi_\kF$, the gluing map $\Theta$ can be presented by a pair of matrices $(\Theta_0, \Theta_\infty)$ of the same size:
$$
\xymatrix{
\jmath^*(\CC^n) \ar[rr]^-\Theta \ar[d]_-\cong & & \imath^*\kF \ar[d]^-{\xi_{\kF}} \\
\CC^n \times \CC^n \ar[rr]^{(\Theta_0,\; \Theta_\infty)} & & \CC^n \times \CC^n.
}
$$
The definition of morphisms in the comma--category is straightforward.

\smallskip
\noindent
For any vector bundle $\kP$ on $E$, let  $\Theta_\kP: \jmath^*\bigl(\kP\big|_{s}\bigr) \lar \imath^* \bigl(\nu^*\kP\bigr)$ be the canonical isomorphism.
It turns out that the functor
\begin{equation}\label{E:equivcat}
\VB(E) \lar \mathsf{Comma}(E), \quad  \kP \mapsto \bigl(\nu^*\kP, \kP\big|_{s}, \Theta_\kP\bigr)
\end{equation}
 is fully faithful. Moreover,   its essential image consists precisely of those triples
$\bigl(\kF, \CC^n, \Theta\bigr)$, for which the gluing morphism $\Theta$ is an isomorphism \cite{Thesis, Survey}. In this way, one can reduce the study of vector bundles on singular curves (in particular, on degenerate elliptic curves) to certain \emph{matrix problems}; see for instance \cite{DrozdGreuel, OldSurvey, Thesis, Survey, BodnarchukDrozd,BK4}.

\begin{definition}\label{D:blowup} Let $c, d \in \NN$ be mutually prime and $n = c+d$. In what follows, we present any matrix $X \in \Mat_{n\times n}(\CC)$ in the block form
$
X = \left(
\begin{array}{c|c}
A & B \\
\hline
C & D
\end{array}
\right),
$
 where   $A$ and $D$ are square matrices of sizes $ c\times c$  and $d \times d$ respectively. We define the matrix $K_{(c,d)} \in \Mat_{n\times n}(\CC)$ by the following recursive procedure.
\begin{enumerate}
\item For $(c, d) = (1,1)$, we  put $K_{(1,1)} = \left(\begin{array}{c|c} 0 & 1 \\ \hline 1 & 0\end{array}\right)$.
\item If $K_{(c,d)} = \left(\begin{array}{c|c} K_1 & K_2 \\ \hline K_3 & K_4\end{array}\right)$ then we write:
$$
K_{(c+d, d)} =
\left(
\begin{array}{cc|c}
K_1 & K_2 & 0 \\
0 & 0 & I_d \\
\hline
K_3 & K_4 & 0
\end{array}
\right) \quad \mbox{and} \quad
K_{(c, d+c)} =
\left(
\begin{array}{c|cc}
0 & 0 & I_c \\
\hline
K_3 & K_4 & 0 \\
K_1 & K_2 & 0
\end{array}
\right),
$$
where $I_d$ (respectively, $I_c$) is the square matrix of size $d \times d$ (respectively, $c\times c$).
\end{enumerate}
\end{definition}

\begin{lemma} For any $e, d\in \NN$ such that $\gcd(c,d) = 1$, we have:
$$
K_{(c,d)} =
\left(
\begin{array}{cc}
0 & I_c \\
I_d & 0
\end{array}
\right) =
\left(
\begin{array}{cccccc}
0 & \dots & 0 & 1 & \dots & 0 \\
\vdots &  & \vdots  & & \ddots & \\
0 & \dots & 0 & 0 & \dots & 1 \\
1 & \dots & 0 & 0 & \dots & 0 \\
\vdots & \ddots & \vdots &  \vdots  & & \vdots \\
0 & \dots & 1 & 0 & \dots & 0
\end{array}
\right).
$$
Next,
$
J_{(c,d)} := K_{(c,d)}^{-1} = \left(
\begin{array}{cc}
0 & I_d \\
I_c & 0
\end{array}
\right)
$
and  for any
$
X = \left(
\begin{array}{c|c}
A & B \\
\hline
C & D
\end{array}
\right)
\in \Mat_{n\times n}(\CC)
$
we have:
\begin{equation}\label{E:Xsharp}
X^\sharp:= J_{(c,d)}\cdot  X \cdot J_{(c,d)}^{-1} = \left(
\begin{array}{c||c}
D & C \\
\hline\hline
B & A
\end{array}
\right).
\end{equation}
\end{lemma}
\begin{proof}
Straightforward computation.
\end{proof}

\begin{proposition}\label{P:triples} Let $c, d \in \NN$ be mutually prime, $n = c + d$ and $\kF = \kO_{\PP^1}^{\oplus c} \oplus  \bigl(\kO_{\PP^1}(1)\bigr)^{\oplus d}$.
Then the following results are true:
\begin{enumerate}
\item The triple $\bigl(\kF, \CC^n, (I_n, K_{(c,d)})\bigr)$ corresponds to a simple vector bundle $\kP$ on $E$ of rank $n$ and degree $d$.
\item
Let $\mathrm{Ad}_{K_{(c,d)}}:  \lieg \lar \lieg$ be given by the formula $Y  \mapsto K_{(c,d)} \cdot Y \cdot  K_{(c,d)}^{-1}$.
Then the  triple $\bigl(\mathit{Ad}_{\PP^1}(\kF), \lieg, \bigl(\mathrm{Id}, \mathrm{Ad}_{K_{(c,d)}}\bigr)\bigr)$ corresponds to the sheaf of Lie algebras $\mathit{Ad}_E(\kP)$.
\item Let $E \stackrel{\sigma}\lar  E$ be the involution corresponding to the automorphism $$\PP^1 \lar \PP^1,\; (a: b) \mapsto (b: a).$$ Then the triple
$\bigl(\kF, \CC^n, (I_n, J_{(c,d)})\bigr)$ corresponds to the  simple vector bundle $\sigma^*(\kP)$.
\end{enumerate}
\end{proposition}

\begin{proof} The first statement is essentially proven in \cite[Theorem 5.1.19]{BK4}. However,  one should mention that in \cite[Algorithm 5.1.20]{BK4}, a slightly different  rule for
the blow--ups from Definition \ref{D:blowup} was chosen. Namely, for a given $K_{(c,d)} = \left(\begin{array}{c|c} K_1 & K_2 \\ \hline K_3 & K_4\end{array}\right)$, it was put  $\widetilde{K}_{(c+d,d)} = K_{(c+d, d)}$, whereas
$$
\widetilde{K}_{(c, d+c)} :=
\left(
\begin{array}{c|cc}
0 & I_c & 0 \\
\hline
K_1 & 0 & K_2 \\
K_3 & 0 & K_4
\end{array}
\right).
$$
It follows from the matrix identity
$$
\left(
\begin{array}{ccc}
I_c & 0 & 0 \\
0 & 0 & I_d \\
0 & I_c & 0
\end{array}
\right)
\cdot
\widetilde{K}_{(c, d+c)}\cdot
\left(
\begin{array}{ccc}
I_c & 0 & 0 \\
0 & 0 & I_c \\
0 & I_d & 0
\end{array}
\right) = {K}_{(c, d+c)}
$$
that the triples $\bigl(\kF, \CC^n, (I_n, K_{(c,d)})\bigr)$ and $\bigl(\kF, \CC^n, (I_n, \widetilde{K}_{(c,d)})\bigr)$ are isomorphic, hence
define isomorphic simple bundle of rank $n$ and degree $d$ on the curve $E$.

\smallskip
\noindent
For the second statement, see for instance \cite[Proposition 6.3]{BH}. Finally, it follows from the definition of the equivalence of categories (\ref{E:equivcat}) and the made choices of  trivializations (\ref{E:trivialization})  that
the vector bundle $\sigma^*(\kP)$ is described by the triple $\bigl(\kF, \CC^n, (K_{(c,d)}, I_n)\bigr)$. It remains to observe that
$$
\bigl(\kF, \CC^n, (K_{(c,d)}, I_n)\bigr) \cong \bigl(\kF, \CC^n, (I_n, K_{(c,d)}^{-1})\bigr) = \bigl(\kF, \CC^n, (I_n, J_{(c,d)})\bigr)
$$
in the category $\mathsf{Comma}(E)$, implying  the third statement. \end{proof}

\begin{corollary}\label{C:SheafA}
Let $0 < d < n$ be mutually prime and $\kP$ be a simple vector bundle of rank $n$ and degree $d$ on a nodal Weierstra\ss{} cubic $E$ and $\kA = \mathit{Ad}_E(\kP)$. Then $\kA = \kA_{(n,d)}$ is described by the triple $\Bigl(\mathit{Ad}_{\PP^1}(\kF), \lieg, \bigl(\mathrm{Id}, \mathrm{Ad}_{J_{(c,d)}}\bigr)\Bigr)$,
where  $\kF = \kO_{\PP^1}^{\oplus c} \oplus  \bigl(\kO_{\PP^1}(1)\bigr)^{\oplus d}$ and $c = n-d$. Moreover, $H^0(E, \kA) = 0 = H^1(E, \kA)$.
\end{corollary}
\begin{proof} Let $\kQ$ be another simple vector bundle on $E$ of rank $n$ and degree $d$. Then there exists a line bundle $\kL \in \Pic^0(E)$ such that
$\kQ \cong \kP \otimes \kL$; see for instance \cite{Burban1}. Therefore, we have isomorphism of sheaves of Lie algebras:
$$
\mathit{Ad}_E(\kQ) \cong \mathit{Ad}_E(\kP) \cong \mathit{Ad}_E\bigl(\sigma^*(\kP)\bigr).
$$
It follows from the third part of Proposition \ref{P:triples} that the sheaf of Lie algebras $\kA$ does not depend on the choice of a simple $\kP$ and is given by the triple $\Bigl(\mathit{Ad}_{\PP^1}(\kF), \lieg, \bigl(\mathrm{Id}, \mathrm{Ad}_{J_{(c,d)}}\bigr)\Bigr)$.
Consider the long cohomology sequence attached to the short exact sequence (\ref{E:DefAd}):
$$
0  \lar H^0(\kA) \lar \End_E(\kP) \stackrel{n\cdot\,}\lar H^0(\kO)   \lar H^1(\kA) \lar \Ext^1(\kP, \kP)  \lar H^1(\kO) \lar 0.
$$
Since all vector spaces $H^0(\kO), H^1(\kO), \End_E(\kP)$ and $\Ext^1(\kP, \kP)$ are one--dimensional, we get the vanishing $H^0(\kA) = 0 = H^1(\kA).$
\end{proof}

\section{On Lagrangian orders in $\lieg\llbrace z^{-1}\rrbrace \times \lieg$}

\noindent
Let $\lieg = \mathfrak{sl}_n(\CC)$ and $\lieD = \lieg\llbrace z^{-1}\rrbrace \times \lieg$. Then we have the following non--degenerate symmetric $\CC$--bilinear forms on the Lie algebras $\lieg$ and $\lieD$ respectively:
\begin{equation}\label{E:classdouble1}
\bigl(\lieg \times \lieg\bigr)\times \bigl(\lieg \times \lieg\bigr) \xrightarrow{\left\langle\,-\,,\,-\,\right\rangle} \CC, \quad
\bigl\langle(f_1, g_1), (f_2, g_2)\bigr\rangle = \tr(f_1 f_2) - \tr(g_1  g_2)
\end{equation}
and
\begin{equation}\label{E:mainform}
\lieD \times \lieD \xrightarrow{\left\langle\,-\,,\,-\,\right\rangle} \CC, \quad \bigl\langle(F_1, f_1), (F_2, f_2)\bigr\rangle =
\res_0 \left(\mathsf{Tr}(F_1 F_2) \frac{dz}{z}\right) - \tr(f_1 f_2),
\end{equation}
where $\liea\llbrace z^{-1}\rrbrace \stackrel{\mathsf{Tr}}\lar \CC\llbrace z^{-1}\rrbrace$ and  $\liea \stackrel{\mathsf{tr}}\lar \CC$  are  the trace maps.
Let $\lieP$ be the image of the injective morphism of Lie algebras $\lieg[z] \stackrel{\jmath}\lar  \lieD, \, F(z) \mapsto \bigl(F(z), F(0)\bigr).$
It is not difficult to see that $\lieP$ is a Lagrangian Lie subalgebra of $\lieD$ with respect to the form (\ref{E:mainform}), i.e.~$\lieP = \lieP^\perp$.

\smallskip
\noindent
According to  the work \cite{KPSST} (see also \cite{PopStolin} for further elaborations), the quasi--trigonometric solutions of (\ref{E:CYBE}) are parameterized by the following objects.
\begin{definition}\label{D:orders} A vector subspace $\lieW \subset \lieD$ is called \emph{Lagrangian order} transversal to $\lieP$ if the following conditions are satisfied.
\begin{enumerate}
\item $\lieW$ is a Lie subalgebra of $\lieD$, for which there exists $m \in \NN$ such that $$z^{-m} \lieg\llbracket z^{-1}\rrbracket \times \{0\} \subseteq \lieW.$$
\item $\lieW$ is a coisotropic subspace of $\lieD$ with respect to the form (\ref{E:mainform}), i.e.~for any $w_1, w_2 \in \lieW$ we have:
$\bigl\langle w_1, w_2\bigr\rangle = 0$.
\item We have a direct sum decomposition $\lieD = \lieP \dotplus \lieW$, i.e.~$\lieD = \lieP + \lieW$ and $\lieP \cap \lieW = 0$.
\end{enumerate}
\end{definition}

\begin{example}\label{Ex:BasicOrder}
Let $\lieg = \lien_- \oplus  \lieh \oplus  \lien_+$ be the standard triangular decomposition of $\lieg$ and
$\Delta_{\lieh} := \bigl\{(h , -h) \,\big|\, h \in \lieh \bigr\} \subset \lieg \times \lieg$. Then the vector space
\begin{equation}\label{E:firstorder}
\lieW_\circ := z^{-1} \lieg \llbracket z^{-1}\rrbracket \times \{0\} + \bigl(\lien_- \times \{0\} + \{0\} \times \lien_+ + \Delta_{\lieh}\bigr)
\end{equation}
is a Lagrangian order transversal to $\lieP$.
It what follows, we call
$
\lieD = \lieW_\circ \dotplus \lieP
$
the \emph{standard decomposition} of $\lieD$.
\end{example}

\begin{lemma}
Let $\lieW \subset \lieD$ be a vector subspace satisfying the second and the third properties of Definition \ref{D:orders}. Then we have:
$\lieW = \lieW^\perp$, i.e.~$\lieW$ is a Lagrangian subspace of $\lieD$.
\end{lemma}
\begin{proof}
By the second assumption of Definition \ref{D:orders}, we have the inclusion  $\lieW \subseteq \lieW^\perp$. Let $u \in \lieW^\perp$. Then there exist uniquely determined $v \in \lieW$ and $l \in \lieP$ such that $u = l + v$. For any $w \in \lieW$ we have:
$
0 = \langle u, w\rangle =  \langle l, w\rangle + \langle v, w\rangle = \langle l, w\rangle.
$
Since $\lieP$ is a coisotropic  subspace of $\lieD$, we have $\bigl\langle l, \,-\,\bigr\rangle\Big|_{\lieP} = 0$. From what was said above follows, that
$\bigl\langle l, \,-\,\bigr\rangle\Big|_{\lieW} = 0$, too. Since the form (\ref{E:mainform}) is non--degenerate, we conclude that $l = 0$, therefore
$u = v \in \lieW$. Hence, $\lieW^\perp \subseteq \lieW$, implying the result.
\end{proof}

\begin{definition}
For any $\sigma \in \Aut_{\CC[z]}\bigl(\lieg[z]\bigr)$,  let $\sigma_\circ  \in \Aut_\CC(\lieg)$ be the induced automorphism.
Then we have a $\CC$--linear automorphism
$$\lieD \stackrel{\widetilde\sigma}\lar \lieD, \; (F, f) \mapsto
\bigl(\sigma(F), \sigma_\circ(f)\bigr),$$ preserving the  bilinear form (\ref{E:mainform}). If $\lieW$ is a Lagrangian order transversal to $\lieP$ then
$\widetilde\sigma(\lieW)$ is also such an order. Conversely,
we call two Lagrangian orders $\lieW'$ and $\lieW''$  \emph{gauge equivalent} if there exists $\sigma \in \Aut_{\CC[z]}\bigl(\lieg[z]\bigr)$ such that $\lieW'' = \widetilde\sigma(\lieW')$.
\end{definition}

\noindent
The next natural question to clarify is the following: in what terms can one classify (up to the gauge equivalence) all Lagrangian orders $\lieW \subset \lieD$?

\begin{definition} For any $c, d \in \NN$ such that $c+d = n$, consider the matrix
$$
T_{(c, d)} = \left(
\begin{array}{c|c}
I_c & 0 \\
\hline
0 & z I_d
\end{array}
\right) \in \GL_n\bigl(\CC[z, z^{-1}]\bigr).
$$
Then we have the following $\CC$--linear Lie algebra automorphisms:
\begin{equation*}
\lieg\llbrace z^{-1}\rrbrace \xrightarrow{\mathrm{Ad}_{(c,d)}}\lieg\llbrace z^{-1}\rrbrace, \; F \mapsto T_{(c,d)} F T_{(c,d)}^{-1}
\quad
\mbox{\rm and}
\quad
\lieD \xrightarrow{\overline{\mathrm{Ad}}_{(c,d)}} \lieD, \; (F, f) \mapsto \bigl(\mathrm{Ad}_{(c,d)}(F), f\bigr).
\end{equation*}
Finally, we put:
$
\lieO_{(c,d)} := \mathrm{Ad}_{(c,d)}\bigl(\lieg\llbracket z^{-1}\rrbracket\bigr) \times \lieg.
$
\end{definition}

\begin{lemma}\label{L:basicsOrders} The following results are true.
\begin{enumerate}
\item The $\CC$--vector space $\lieO_{(c,d)}$ is a Lie subalgebra of $\lieD$.
\item The orthogonal complement  of $\lieO_{(c,d)}$ with respect to the bilinear form (\ref{E:mainform}) has the following description:
$$
\lieO_{(c,d)}^\perp = \mathrm{Ad}_{(c,d)}\bigl(z^{-1}\lieg\llbracket z^{-1}\rrbracket\bigr) \times \{0\} =
\left(z^{-2} \lieg\llbracket z^{-1}\rrbracket + z^{-1} \lieb_{(c,d)} + \lien_{(c,d)} \right) \times \{0\}.
$$
\item Moreover, $\lieO_{(c,d)}^\perp \subset \lieO_{(c,d)}$ is a Lie ideal and the linear map
$$
\lieO_{(c,d)}/\lieO_{(c,d)}^\perp \stackrel{\imath}\lar \lieg \times \lieg, \quad \overline{\bigl(\mathrm{Ad}_{(c,d)}(F), f\bigr)}\mapsto
\bigl(F(0), f\bigr)
$$
is an isomorphism of Lie algebras.
\item Finally, the bilinear form (\ref{E:mainform}) defines a non--degenerate bilinear form on the quotient $\lieO_{(c,d)}/\lieO_{(c,d)}^\perp$ and the morphism
$\imath$ is an isometry.
\end{enumerate}
\end{lemma}
\begin{proof}
Let $\lieO = \lieg\llbracket z^{-1}\rrbracket \times \{0\}$. It is clear that $\lieO$ is a Lie subalgebra of $\lieD$. Moreover,
$\lieO = z^{-1} \lieg\llbracket z^{-1}\rrbracket \times \{0\}$ and  $\lieO^\perp \subset \lieO$ is a Lie ideal. Since the automorphism $\overline{\mathrm{Ad}}_{(c,d)}$ preserves the form (\ref{E:mainform}), the first three statement follow. It is also clear that
the bilinear form on $\lieO_{(c,d)}/\lieO_{(c,d)}^\perp$ given by the formula
$$
\bigl\langle \overline{G}_1, \overline{G}_2\bigr\rangle := \bigl\langle {G}_1, {G}_2\bigr\rangle\quad \mbox{\rm for any}\;
G_1, G_2 \in \lieO_{(c,d)}
$$
is well--defined. Let $G_i = \bigl(\mathrm{Ad}_{(c,d)}(F_i), f_i\bigr)$ for some $F_i \in \lieg\llbracket z^{-1}\rrbracket$ and
$f_i \in \lieg$, where $i = 1,2$. Then we have:
$$
\bigl\langle {G}_1, {G}_2\bigr\rangle = \res_0 \left(\mathsf{Tr}\bigl(\mathrm{Ad}_{(c,d)}(F_1 F_2)\bigr) \frac{dz}{z}\right) - \tr(f_1 f_2) = \tr\bigl(F_1(0) F_2(0)\bigr) - \tr(f_1 f_2),
$$
implying the fourth statement.
\end{proof}

\begin{definition} Consider the following vector space:
\begin{equation*}
\nabla_{(c,d)} :=  \left\{
\left(
\begin{array}{c|c}
A & 0 \\
\hline
C' & D
\end{array}
\right),
\left(
\begin{array}{c|c}
A & 0 \\
\hline
C'' & D
\end{array}
\right)\; \left| \;
\left(
\begin{array}{c|c}
A & 0 \\
\hline
C' & D
\end{array}
\right),
\left(
\begin{array}{c|c}
A & 0 \\
\hline
C'' & D
\end{array}
\right) \in \lieb_{(c,d)}
\right.
\right\} \subseteq \lieg \times \lieg,
\end{equation*}
where $\lieb_{(c,d)} \subset \lieg$ is the Lie algebra defined by (\ref{E:solieb}).
\end{definition}

\begin{lemma}\label{L:AlgebraNabla} The following results are true.
\begin{enumerate}
\item The vector space $\nabla_{(c,d)}$ is a Lagrangian Lie subalgebra of $\lieg \times \lieg$ with respect to the bilinear form (\ref{E:classdouble1}).
\item Consider the Lie algebra  $\lieP_{(c,d)}:= \lieP \cap \lieO_{(c,d)}$. Then we have:
$$
\lieP_{(c,d)} = \left\{(X, X) \big| X \in \lieb_{(c,d)}\right\} + z \lien_{(c,d)} \times \{0\}.
$$
\item Finally, the image of $\lieP_{(c,d)}$ under the homomorphism of Lie algebras
$$
\lieP_{(c,d)} \lar \lieP_{(c,d)}/\bigl(\lieP_{(c,d)}\cap \lieO_{(c,d)}\bigr) \stackrel{\imath}\lar \lieg \times \lieg
$$
is the Lie algebra $\nabla_{(c,d)}$.
\end{enumerate}
\end{lemma}

\begin{proof}
It follows  from the definition that $\nabla_{(c,d)}$ is a coisotropic subspace of $\lieg \times \lieg$. Since
$\dim_{\CC}\bigl(\nabla_{(c,d)}\bigr) = \dim_{\CC}(\lieg)$, we get: $\nabla_{(c,d)} = \nabla_{(c,d)}^\perp$. The proof of the second statement is straightforward. From this  description of $\lieP_{(c,d)}$ it follows  that $\lieP_{(c,d)}/\bigl(\lieP_{(c,d)}\cap \lieO_{(c,d)}\bigr) = $
$$
\left\{
\overline{\left(
\mathrm{Ad}_{(c,d)}\left(
\begin{array}{c|c}
A & 0 \\
\hline
C' & D
\end{array}
\right),
\left(
\begin{array}{c|c}
A & 0 \\
\hline
C'' & D
\end{array}
\right)
\right)}\; \left| \;
\left(
\begin{array}{c|c}
A & 0 \\
\hline
C' & D
\end{array}
\right),
\left(
\begin{array}{c|c}
A & 0 \\
\hline
C'' & D
\end{array}
\right) \in \lieb_{(c,d)}
\right.
\right\},
$$
implying the last statement.
\end{proof}

\noindent
The following result  plays the main role in the classification of Lagrangian orders; see  \cite[Theorem 6]{KPSST} and the references therein.
\begin{theorem}\label{T:classifOrders}
Let $\lieW \subset \lieD$ be a Lagrangian order transversal to $\lieP$.  Then there exists (in general, non--unique) automorphism $\sigma \in \Aut_{\CC[z]}\bigl(\lieg[z]\bigr)$ as well as
$c,d \in \NN_0$ satisfying  $c+d = n$,   such that
 $\widetilde\sigma\bigl(\lieW\bigr) \subseteq \lieO_{(c,d)}$.
\end{theorem}

\begin{remark}
If $\lieW$ is a Lagrangian order such that $\lieW \subseteq \lieO_{(c,d)}$ then we have  inclusions:
\begin{equation}
\lieO_{(c,d)}^\perp \subseteq \lieW^\perp = \lieW \subseteq \lieO_{(c,d)}.
\end{equation}
\end{remark}

\begin{proposition}\label{P:reduction} There exists a bijection between the following two sets:
\begin{enumerate}
\item Lagrangian orders $\lieW \subset \lieO_{(c,d)}$ transversal to $\lieP$.
\item Lagrangian subalgebras $\liew \subset \lieg \times \lieg$ such that $\lieg \times \lieg = \liew \dotplus \nabla_{(c,d)}$.
\end{enumerate}
\end{proposition}
\begin{proof}
For any Lagrangian order $\lieW \subset \lieO_{(c,d)}$ we put: $\liew := \imath\bigl(\lieW/\lieO_{(c,d)}^\perp\bigr) \subset \lieg \times \lieg$.
Since $\lieD = \lieW \dotplus \lieP$, we get:
$
\lieO_{(c,d)} = \lieW \dotplus \lieP_{(c,d)}.
$
As a consequence, we obtain:
\begin{equation}\label{E:splitting}
\lieO_{(c,d)}/\lieO_{(c,d)}^\perp = \lieW/\lieO_{(c,d)}^\perp \dotplus \lieP_{(c,d)}/\bigl(\lieP_{(c,d)}\cap \lieO_{(c,d)}^\perp\bigr).
\end{equation}
Applying to both sides of (\ref{E:splitting}) the isomorphism $\imath$ from Lemma \ref{L:basicsOrders} and taking into account   Lemma \ref{L:AlgebraNabla}, we get a direct sum decomposition
$\lieg \times \lieg = \liew \dotplus \nabla_{(c,d)}$.
It follows from Lemma \ref{L:basicsOrders} that $\liew$ is a coisotropic subspace of $\lieg \times \lieg$. By the dimension reasons, we have:
$\liew = \liew^\perp$.

\smallskip
\noindent
Conversely, let $\lieg \times \lieg = \liew \dotplus \nabla_{(c,d)}$ be a Lagrangian decomposition. Then we put:
$$
\lieW := \imath^{-1}(\liew) = \lieO_{(c,d)}^\perp \dotplus \left\{\left(\mathrm{Ad}_{(c,d)}(f_1), f_2\right)\Big| (f_1, f_2) \in \liew \right\}.
$$
It follows that $\lieO_{(c,d)} = \lieW \dotplus \lieP_{(c,d)}$ and $\lieW$ is a coisotropic subspace of $\lieO_{(c,d)}$ (hence, of $\lieD$). Since
$\lieO_{(c,d)} + \lieP = \lieD$ and $\lieP \cap \lieO_{(c,d)}^\perp = 0$, we get a direct sum decomposition
$\lieD = \lieW \dotplus \lieP$, satisfying all assumptions of Definition \ref{D:orders}.
\end{proof}

\begin{corollary} Proposition \ref{P:reduction} implies that the classification of Lagrangian orders in the sense of Definition \ref{D:orders} (and as a consequence, the classification of
  quasi--trigonometric  solutions of
CYBE for the Lie algebra $\lieg$) reduces to a finite dimensional problem of the description  of Lagrangian decompositions $\lieg \times \lieg = \liew \dotplus \nabla_{(c,d)}$ for all pairs
$e,d \in \NN_0$ such that $e+d = n$.
\end{corollary}

\begin{definition}\label{D:twisteddiag} For any $c,d \in \NN$ such that $\gcd(c,d) = 1$ and $c + d = n$, let
\begin{equation}\label{E:twisteddiag}
\Delta_{(c,d)} :=
\left\{(X, X^\sharp) \,\big|\, X \in \lieg \right\} = \left\{
\left(
\begin{array}{c|c}
A & B \\
\hline
C & D
\end{array}
\right),
\left(
\begin{array}{c||c}
D & C \\
\hline
\hline
B & A
\end{array}
\right)
\right\} \subseteq \lieg \times \lieg
\end{equation}
be the \emph{twisted diagonal}.
\end{definition}

\begin{proposition}\label{P:maindecomp}  In the notation as above, we have a Lagrangian decomposition
\begin{equation}\label{E:maindecomp}
\lieg \times \lieg = \Delta_{(c,d)} \dotplus \nabla_{(c,d)}.
\end{equation}
\end{proposition}
\begin{proof}
It is clear that  $\Delta_{(c,d)}$ is a coisotropic (hence Lagrangian) Lie subalgebra of $\lieg \times \lieg$ with respect to the form (\ref{E:classdouble1}).  Hence, it is sufficient to show that $\Delta_{(c,d)} \cap \nabla_{(c,d)} = 0$.
Instead of proving this statement  directly, we use a geometric argument.

\smallskip
\noindent
Let $\kA_{(n,d)}$ be the sheaf of Lie algebras on a nodal Weierstra\ss{} curve $E$, introduced in Corollary \ref{C:SheafA}.
Then we have: $\Gamma\bigl(E, \kA_{(n,d)}\bigr) = 0$. Moreover,
$$
\Gamma\bigl(\PP^1, \mathit{Ad}_{\PP^1}(\kF)\bigr) =
\left\{
\left(
\begin{array}{c|c}
A & 0 \\
\hline
C'z_0 + C''z_1 & D
\end{array}
\right) \left| \;\mbox{\rm where}\;
\left(
\begin{array}{c|c}
A & 0 \\
\hline
C' & D
\end{array}
\right), \left(
\begin{array}{c|c}
A & 0 \\
\hline
C'' & D
\end{array}
\right) \in \lieb_{(c,d)} \right.
\right\},
$$
where as usual, $\kF = \kO_{\PP^1}^{\oplus c} \oplus  \bigl(\kO_{\PP^1}(1)\bigr)^{\oplus d}$. It follows from the definition of the functor (\ref{E:equivcat}) that the structure sheaf $\kO_E$ corresponds to the triple $\Bigl(\kO_{\PP^1}, \CC, \bigl((1), (1)\bigr)\Bigr)$. Therefore,
$$
\Hom_E(\kO_E, \kA) = \Hom_{\mathsf{Comma(E)}}\Bigl(\bigl(\kO_{\PP^1}, \CC, \bigl((1), (1)\bigr)\bigr), \bigl(\mathit{Ad}_{\PP^1}(\kF), \lieg, \bigl(\mathrm{Id}, \mathrm{Ad}_{J_{(c,d)}}\bigr)\bigr)\Bigr) = 0.
$$
This condition precisely means that whenever
$$\left(
\begin{array}{c|c}
A & 0 \\
\hline
C''& D
\end{array}
\right) =
J_{(c,d)} \cdot
\left(
\begin{array}{c|c}
A & 0 \\
\hline
C'& D
\end{array}
\right)\cdot
J_{(c,d)}^{-1} = \left(
\begin{array}{c||c}
D & C' \\
\hline
\hline
0 & A
\end{array}
\right)
$$
for some $\left(
\begin{array}{c|c}
A & 0 \\
\hline
C' & D
\end{array}
\right), \left(
\begin{array}{c|c}
A & 0 \\
\hline
C'' & D
\end{array}
\right) \in \lieb_{(c,d)}$ then $\left(
\begin{array}{c|c}
A & 0 \\
\hline
C' & D
\end{array}
\right) =  \left(
\begin{array}{c|c}
A & 0 \\
\hline
C'' & D
\end{array}
\right) = 0$. Therefore, $\Delta_{(c,d)} \cap \bigl(J_{(c,d)} \nabla_{(c,d)} J_{(c,d)}^{-1}\bigr) = 0$, implying the statement.
\end{proof}

\begin{theorem}\label{T:CremmerGervaisOrder} The following vector space
\begin{equation}\label{E:CremmerGervaisOrder}
\lieW_{(c,d)}:=  \left(z^{-2} \lieg\llbracket z^{-1}\rrbracket + z^{-1} \lieb_{(c,d)} + \lien_{(c,d)} \right) \times \{0\} + \widetilde{\Delta}_{(c,d)} \subset \lieD,
\end{equation}
where
\begin{equation}
\widetilde{\Delta}_{(c,d)} :=
\left\{
\left(
\left(
\begin{array}{c|c}
A & z^{-1}B \\
\hline
zC & D
\end{array}
\right),
\left(
\begin{array}{c||c}
D & C \\
\hline\hline
B & A
\end{array}
\right)
\right)
\left| \; \mbox{\rm for all} \;  \left(
\begin{array}{c|c}
A & B \\
\hline
C & D
\end{array}
\right) \in \lieg \right.
\right\} \subset \lieD,
\end{equation}
is a Lagrangian order in $\lieD$ transversal to $\lieP$.
\end{theorem}

\begin{proof} The vector space $\lieW_{(c,d)}$ is precisely the Lagrangian order  corresponding to   the twisted diagonal $\Delta_{(c,d)}$ defined by
(\ref{E:twisteddiag}).
\end{proof}

\section{On quasi--trigonometric solutions of CYBE}

\noindent
Let $\lieW$ be a Lagrangian order in $\lieD$, transversal to $\lieP$; see Definition \ref{D:orders}.
We begin with a description of the procedure, assigning to $\lieW$ the corresponding quasi--trigonometric solution $r_{\lieW}$ of the classical Yang--Baxter equation (\ref{E:CYBE}).

Let $\bigl\{g_\beta\bigr\}_{\beta \in \Phi}$ be a basis of the Lie algebra   $\lieg$, where $\Phi$ is  an  index set. For any
$(k, \beta) \in \widehat\Phi := \NN_0 \times \Phi$, we denote by $g_{(k, \beta)} = \jmath\bigl(g_\beta z^k\bigr) \in \lieP$. Obviously,
$\bigl\{g_{(k, \beta)}\bigr\}_{(k, \beta) \in \widehat\Phi}$ is a basis of the  Lie algebra $\lieP$.  Let $\bigl\{w_{(k, \beta)}\bigr\}_{(k, \beta) \in \widehat\Phi}$ be the corresponding dual topological basis
of $\lieW$, i.e. $\bigl\langle g_{(k,\beta)}, w_{(k', \beta')}\bigr\rangle = \delta_{k, k'} \cdot \delta_{\beta, \beta'}$ for all $k ,k' \in \NN_0$ and $\beta, \beta' \in \Phi$.

\smallskip
\noindent
Let $\lieD \stackrel{\pi}\lar  \lieg\llbrace z^{-1}\rrbrace$ be the canonical projection.
For the following result, see \cite[Theorem 2 and Theorem 4]{KPSST}.

\begin{theorem}
 The  formal power series
\begin{equation*}
r_{\lieW}(x,y):= \sum\limits_{(k, \beta)\in \widehat\Phi} \pi(g_{(k, \beta)})\Big|_x \otimes \pi\bigl(w_{(k, \beta)}\bigr)\Big|_y = \sum\limits_{(k, \beta)\in \widehat\Phi} x^k g_\beta \otimes \pi\bigl(w_{(k, \beta)}\bigr)\Big|_y \in \lieg[x] \widehat\otimes \lieg\llbrace y^{-1}\rrbrace
\end{equation*}
does not depend on the initial choice of basis $\bigl\{g_\beta\bigr\}_{\beta \in \Phi}$ of the Lie algebra $\lieg$  and
defines  a quasi--trigonometric solution of (\ref{E:CYBE}). Moreover, for any $\sigma \in \Aut_{\CC[z]}\bigl(\lieg[z]\bigr)$ we have:
$$
r_{\widetilde{\sigma}(\lieW)}(x, y) = \bigl(\sigma(x) \otimes \sigma(y)\bigr)\bigl(r_{\lieW}(x,y)\bigr).
$$
\end{theorem}
We have the standard triangular decomposition $\lieg = \lien_+ \oplus \lieh \oplus \lien_-$ of the Lie algebra $\lieg$ into a direct sum of its  Lie subalgebras, consisting  of upper triangular, diagonal and lower diagonal matrices respectively. In these terms, we have: $\Phi = \Phi_+ \cup \Phi_0 \cup \Phi_-$, where $\Phi_+  := \bigl\{(i, j) \in \NN^2 \,\big|\,  1 \le i < j \le n\bigr\}$ denotes the set of positive roots of $\lieg$, $\Phi_-  := \bigl\{(i, j) \in \NN^2 \,\big|\,  1 \le j < i \le n\bigr\}$ is  the set of negative roots and $\big|\Phi_0\big| = n-1$. Let $\{h_i\}_{i \in \Phi_0}$ be some (not necessary orthogonal) basis of $\lieh$. For any $\beta \in \Phi$ we put:
\begin{equation}\label{E:basisG}
g_\beta =
\left\{
\begin{array}{ccl}
e_{\pm \alpha} & \mbox{\rm if} &\beta = \pm \alpha  \in \Phi_\pm \\
h_i & \mbox{\rm if} &\beta = i \in \Phi_0.
\end{array}
\right.
\end{equation}
Let $\lieD = \lieP \dotplus \lieW_\circ$ be the standard  triangular decomposition, where $\lieW_\circ$ is the order introduced in  Example \ref{Ex:BasicOrder}, and $\left\{w_{(k, \beta)}^\circ\right\}_{(k, \beta) \in \widehat\Phi}$ be the  topological basis of $\lieW^\circ$, dual to
$\left\{g_{(k, \beta)}\right\}_{(k, \beta) \in \widehat\Phi}$.
 Then the following formulae are true.
$$
\left\{
\begin{array}{llccl}
w_{(k,\beta)}^\circ  & = & (z^{-k} g_{\beta}^\ast, 0)& \mbox{\rm  for any} &  k \ge 1 \; \mbox{\rm and} \;  \beta \in \Phi\\
w_{(0,\alpha)}^\circ  & = & (e_{-\alpha}, 0)& \mbox{\rm  for any} & \alpha \in \Phi_+ \\
w_{(0,-\alpha)}^\circ  & = & (0, -e_{\alpha})& \mbox{\rm  for any} & \alpha \in \Phi_+ \\
w_{(0,i)}^\circ & = & (\frac{1}{2} h_i^\ast, -\frac{1}{2} h_i^\ast)& \mbox{\rm  for any} & i \in \Phi_0. \\
\end{array}
\right.
$$
Since $\gamma = \sum\limits_{\beta \in \Phi} g_\beta \otimes g_\beta^\ast \in \lieg \otimes \lieg$ is the Casimir element of $\lieg$, we have:
$$
\sum\limits_{k = 1}^\infty \sum\limits_{\beta \in \Phi} x^k g_\beta \otimes y^{-k} g_{\beta}^\ast = \frac{x}{y-x} \gamma.
$$
On the other hand, we have the formula:
$$
\sum\limits_{\beta \in \Phi} g_\beta \otimes \pi\bigl(w_{(0, \beta)}^\circ\bigr) = \frac{1}{2}\left(\gamma + \sum\limits_{\alpha \in \Phi_+} e_\alpha \wedge
e_{-\alpha}\right).
$$
Therefore, the series $r_{\mathsf{st}}:= r_{\lieW_\circ}$ corresponding to the order $\lieW_\circ$, has the form:
\begin{equation}\label{E:DrinfeldJimbo}
r_{\mathsf{st}}(x, y) = \frac{1}{2}\left(
\frac{y+x}{y-x} \gamma + \sum\limits_{\alpha \in \Phi_+} e_\alpha \wedge e_{-\alpha}\right).
\end{equation}
It is a well---known (called in what follows \emph{standard}) quasi--trigonometric solution of CYBE, defining the ``standard'' Lie bialgebra structure
on the Lie algebra $\lieg[z]$ (or, more conventionally, on $\lieg[z, z^{-1}]$ respectively  on the affine Lie algebra $\widehat{\lieg}$; see for instance \cite[Section 6.2.1]{EtingofSchiffmann}).
\begin{lemma} Let $\lieW$ be a Lagrangian order in $\lieD$ transversal to $\lieP$. Then for any $(k, \beta) \in \widehat\Phi$, there exist uniquely determined
 elements $w_{(k, \beta)} \in \lieW$ and $p_{(k, \beta)} \in \lieP$ such that
 $$
 w_{(k, \beta)}^\circ = w_{(k, \beta)} - p_{(k, \beta)} \quad \mbox{\rm for all}\quad (k, \beta) \in \widehat\Phi.
 $$
 Moreover, $\left\{w_{(k, \beta)}\right\}_{(k, \beta) \in \widehat\Phi}$ is a topological basis of $\lieW$ and
$\bigl\langle g_{(k',\beta')}, w_{(k, \beta)}\bigr\rangle = \delta_{k, k'} \cdot \delta_{\beta, \beta'}$ for all $k,k' \in \NN_0$ and $\beta, \beta' \in \Delta$. Finally, we have the following formula for the series $r_{\lieW}$:
\begin{equation}\label{E:FromOrderToCYBE}
r_{\lieW}(x, y) = r_{\mathsf{st}}(x, y) + p_{\lieW}(x,y) = r_{\mathsf{st}}(x, y) + \sum\limits_{(k, \beta) \in \widehat\Phi} x^k g_\beta \otimes p_{(k, \beta)}(y),
\end{equation}
where  $p_{\lieW}(x,y) \in (\lieg \otimes \lieg)[x, y]$.
\end{lemma}
\begin{proof}
Existence and uniqueness of the elements $w_{(k, \beta)} \in \lieW$ and $p_{(k, \beta)} \in \lieP$  for each $(k, \beta) \in \widehat\Phi$ follows from the direct sum decomposition
$\lieD = \lieW \dotplus \lieP$, applied to the element $w_{(k, \beta)}^\circ \in \lieD$.  Since $\lieP$ is a coisotropic subspace of $\lieW$, we have:
$$
\bigl\langle g_{(k',\beta')}, w_{(k, \beta)}\bigr\rangle = \bigl\langle g_{(k',\beta')}, w_{(k, \beta)}^\circ + p_{(k, \beta)}\bigr\rangle = \bigl\langle g_{(k',\beta')}, w_{(k, \beta)}^\circ\bigr\rangle = \delta_{k, k'} \cdot \delta_{\beta, \beta'}.
$$
It follows from the first assumption of Definition \ref{D:orders} that $p_{(k, \beta)} = 0$ for any $(k, \beta)$ such that $k \ge m$. This implies that
$p_{\lieW}(x,y) \in (\lieg \otimes \lieg)[x, y]$.
\end{proof}

\smallskip
\noindent
Our next goal is to derive  a closed formula for the quasi--trigonometric $r$--matrix $r_{(c, d)}$, corresponding to the order $\lieW_{(c, d)}$ defined by formula (\ref{E:CremmerGervaisOrder}).
\begin{definition} Let
$\bar{\Phi} := \bigl\{(i, j) \in \NN^2 \,\big|\, 1 \le i, j \le n \bigr\} \cong \ZZ/n\ZZ \times \ZZ/n\ZZ$.
Then we have a permutation
$
\bar{\Phi} \stackrel{\tau_c}\lar \bar{\Phi}, \; (i, j) \mapsto (i+c, j+c).
$
Note that due to the condition $\gcd(n, c) = 1$, the permutation $\tau_c$ has order $n$.  Abusing the notation, we shall also denote by $\tau_c$ the Lie algebra automorphism
$$
\tau_c = \mathrm{Ad}_{K_{(c,d)}}: \mathfrak{gl}_n(\CC) \lar \mathfrak{gl}_n(\CC), \quad X \mapsto  K_{(c,d)}\cdot X\cdot K_{(c,d)}^{-1},
$$
since $\tau_c\bigl(e_{i,\,j}\bigr) = e_{i+c,\,j+c}$ for any $(i, j) \in \bar\Phi$. Note that in the notation (\ref{E:Xsharp}) we have:
$X = \bigl(\tau_c(X)\bigr)^\sharp$ for any $X \in \mathfrak{gl}_n(\CC)$.
\end{definition}
\begin{lemma}\label{L:trivia} Let $u = e_{11} \in \mathfrak{gl}_n(\CC)$ be the first matrix unit and $I \in \mathfrak{gl}_n(\CC)$ be  the identity matrix.
For any $1 \le i \le n-1$, we put:
\begin{equation}\label{E:basish}
q_{i, c} := \tau_c^{i}(u) - \tau_c^{i-1}(u), \; w_{i, c} = \tau_c^{i-1}(u) - \frac{1}{n}I \;\; \mbox{\rm and}\;\; f_{i, c} := \frac{1}{2}\bigl(\tau_c^{i}(u) + \tau_c^{i-1}(u)\bigr) - \frac{1}{n} I.
\end{equation}
Then the following results are true.
\begin{enumerate}
\item $\left(q_{1, c}, \dots, q_{n-1, c}\right)$ is a basis of the standard Cartan Lie subalgebra $\lieh \subset \lieg$.
\item For any $1 \le i \le n-1$, we have the following identity in $\lieh \times \lieh$:
$$
\frac{1}{2}\left(q_{i, c}, -q_{i, c}\right) = \bigl(\tau_c(w_{i, c}), w_{i, c}\bigr) - (f_{i, c}, f_{i, c}).
$$
\item
Let $\left(q_{1, c}^\ast, \dots, q_{n-1, c}^\ast\right)$ be the dual basis of  $\lieh$ with respect to the trace form. Then
$$
\left\langle(q_{i, c}^\ast, q_{i, c}^\ast), \bigl(\tau(w_{j, c}), w_{j, c}\bigr)\right\rangle = \delta_{ij} \quad \mbox{\rm for all}\quad 1 \le i, j \le n-1.
$$
\end{enumerate}
\end{lemma}
\begin{proof}
Straightforward computation.
\end{proof}

\begin{definition} In the notation of formula (\ref{E:basisG}), we put: $h_{i, c} = q_{i, c}^\ast$ for all $1 \le i \le n-1$.
For any $\alpha \in \Phi_+ = \bigl\{(i, j) \in \bar{\Phi} \,\big|\,  i < j \bigr\}$, we denote:
$
p_c(\alpha) = \min\left\{k \in \NN \, \big|\, \tau_c^k(\alpha) \notin \Phi_+\right\}.
$
\end{definition}

\noindent
The following theorem is the main result of this section.
\begin{theorem}\label{T:maincomputation} Let $r_{(c, d)}$ be the solution of the classical Yang--Baxter equation (\ref{E:CYBE}),
 corresponding to the order $\lieW_{(c, d)}$ defined by formula (\ref{E:CremmerGervaisOrder}). Then we have:
\begin{equation}\label{E:finalformula1}
r_{(c, d)}(x, y) = r_{\mathsf{st}}(x, y) + p_{(c, d)}(x, y),
\end{equation}
where
$
r_{\mathsf{st}}(x, y) = \dfrac{1}{2}\Bigl(\dfrac{y+x}{y-x} \gamma + \sum\limits_{\alpha \in \Phi_+} e_{\alpha} \wedge e_{-\alpha}\Bigr)
$
is the ``standard'' quasi--trigonometric $r$--matrix  and
\begin{equation*}
p_{(c, d)}(x,y) = \sum\limits_{\alpha \in \Phi_+}\Bigl(\Bigl(\sum\limits_{k = 1}^{p(\alpha)-1} e_{\tau^k(\alpha)} \wedge e_{-\alpha}\Bigr) +
x e_{\tau^{p(\alpha)}(\alpha)} \otimes
e_{-\alpha}  - y e_{-\alpha} \otimes e_{\tau^{p(\alpha)}(\alpha)}\Bigr) + \sum\limits_{i=1}^{n-1} q_i^\ast \otimes f_i.
\end{equation*}
\end{theorem}
\begin{remark}
To simplify the notation, we omit  the subscript $c$ when writing  $\tau_c$, $p_c(\alpha)$, $q_{i, c}$, $w_{i, c}$ and $f_{i, c}$.
In the terms of Theorem A from the Introduction, we have: $r_{(c, d)}(x, y) = r_c(x, y)$ and $p_{(c,d)}(x, y) = u_c(x, y) + t_c$.
\end{remark}
\begin{proof} We have formula (\ref{E:FromOrderToCYBE}) for the solution $r_{\lieW}$ corresponding to a Lagrangian order $\lieW$ transversal to $\lieP$.
Therefore, it is sufficient to determine the elements $w_{(k, \beta)}$ and $p_{(k, \beta)}$ for each $(k, \beta) \in \widehat\Phi$. Since
$z^{-2} \lieg \llbracket z^{-1}\rrbracket \times \{0\} \subset \lieW_{(c,d)}$, we have:
$$
w_{(k, \beta)} = w_{(k, \beta)}^\circ \;\;  \mbox{\rm and}\;\;   p_{(k, \beta)} = 0 \;\;  \mbox{\rm  for all} \; \;  (k, \beta) \in \widehat\Phi \;\;  \mbox{\rm such that}\; \;  k \ge 2.
$$
Therefore, it is sufficient to determine $w_{(k, \beta)}$ and $p_{(k, \beta)}$ for $k = 0, 1$.

\smallskip
\noindent
\underline{Claim 1}. For any $\alpha \in \Phi_+$ we have: $w_{(0, -\alpha)}  := \bigl(e_\alpha, e_\alpha)^\ast = $
\begin{equation*}
\begin{split}
- \Bigl(\bigl(e_{\tau(\alpha)}, e_\alpha\bigr) + \bigl(e_{\tau^2(\alpha)}, e_{\tau(\alpha)}\bigr)  + \dots + \bigl(e_{\tau^{p(\alpha)-1}(\alpha)}, e_{\tau^{p(\alpha)-2}(\alpha)}\bigr) + \bigl(z e_{\tau^{p(\alpha)}(\alpha)}, e_{\tau^{p(\alpha)-1}(\alpha)}\Bigr) \\
= (0, -e_{\alpha}) - \bigl(z e_{\tau^{p(\alpha)}(\alpha)}, 0\bigr) - \sum\limits_{k = 1}^{p(\alpha)}\bigl(e_{\tau^k(\alpha)}, e_{\tau^k(\alpha)}\bigr) =
w_{(0, -\alpha)}^\circ + p_{(0, -\alpha)}.
\end{split}
\end{equation*}
Indeed, we have: $p_{(0, -\alpha)} \in \lieP$, the identify $w_{(0, -\alpha)} =  w_{(0, -\alpha)}^\circ + p_{(0, -\alpha)}$ is true and
$w_{(0, -\alpha)} \in \widetilde\Delta_{(c,d)} \subset \lieW_{(c, d)}$. This implies the claim.

\smallskip
\noindent
\underline{Claim 2}. For any $1 \le i \le n-1$ we have: $w_{(0, i)}  := \bigl(h_i^\ast, h_i^\ast)^\ast = \bigl(\tau(w_i), w_i\bigr)$ and
$$
\bigl(\tau(w_i), w_i\bigr) = \frac{1}{2}\bigl(h_i, -h_i\bigr) + \bigl(f_i, f_i\bigr).
$$
This result is a consequence of the third part of Lemma \ref{L:trivia}.

\smallskip
\noindent
Since the tensor $p_{(c,d)}(x, y)$ is skew--symmetric (i.e.~$p_{(c,d)}^{12}(x, y) = -p_{(c,d)}^{21}(y, x)$), Claim 1 and Claim 2 actually imply the statement of the theorem. For the sake of completeness, we shall give a description of the remaining elements of the dual topological basis of $\lieW_{(c, d)}$.

\smallskip
\noindent
Let $\kappa = \tau^{-1}: \overline\Phi \lar \overline\Phi$, i.e.~$\kappa(i, j) = (i + d, j +d)$.
Note that $w_{(1, \beta)} = \bigl(z g_\beta, 0)^\ast = \bigl(z^{-1} g_\beta^\ast, 0\bigr)$ for any $\beta \in \Phi$ but $\beta = -\alpha$ for $\alpha \in \Phi_+$ such that $\kappa(\alpha) \in \Phi_-$.

\smallskip
\noindent
\underline{Claim 3}. Let $\alpha \in \Phi_+$ such that $\kappa(\alpha) \in \Phi_-$. Let
$$
q(\alpha) := \max\bigl\{k \in \NN \, \big|\; \kappa(\alpha), \dots, \kappa^k(\alpha) \in \Phi_-\bigr\}.
$$
Then we have: $w_{(1, -\alpha)}  := \bigl(z e_{-\alpha}, 0)^\ast =
\bigl(z^{-1} e_{\alpha}, 0\bigr) + \sum\limits_{k = 1}^{q(\alpha)} \bigl(e_{\kappa^k(\alpha)}, e_{\kappa^k(\alpha)}\bigr).
$

\smallskip
\noindent
\underline{Claim 4}. For  any $\beta \in \Phi_-$ we put:
$
t(\beta) := \max\bigl\{k \in \NN \, \big|\; \kappa(\beta), \dots, \kappa^k(\beta) \in \Phi_-\bigr\}.
$
Then we have: $w_{(0, \beta)} = \bigl(e_{-\beta}, e_{-\beta}\bigr)^\ast = \bigl(e_\beta, 0) + \sum\limits_{k = 1}^{t(\beta)}
\bigl(e_{\kappa^k(\beta)}, e_{\kappa^k(\beta)}\bigr)$.
\end{proof}

\section{Computation of the geometric $r$--matrix for $\bigl(E, \kA_{(n,d)}\bigr)$}

\smallskip
\noindent
We keep  the notation of Section \ref{S:BundlesNodalCurve}. In particular, $E$ is a nodal Weierstra\ss{} cubic, $\PP^1 \stackrel{\nu}\lar E$ its normalization, $c,d \in \NN$ are mutually prime and such that  $n = c + d$, and  $\kA = \kA_{(n,d)}$. In what follows, we shall identify
$x \in \CC^\ast$ with the corresponding smooth point $\nu(1:x)$ of the curve $E$. In the notation made, $\omega = \dfrac{dz}{z}$ is a generator of the
space $\Gamma(E, \Omega)$, where $\Omega$ is the sheaf of Rosenlicht--regular differential $1$--forms on $E$.

\smallskip
\noindent
For any $x \ne y \in \breve{E}$, consider the linear map $\kA\big|_{x} \xrightarrow{\rho^{\sharp}(x, y)} \kA\big|_{y}$ making the following diagram  of vector spaces
\begin{equation}\label{E:rmatrixandmaps}
\begin{array}{c}
\xymatrix{
 & \Gamma\bigl(X, \kA(x)\bigr) \ar[ld]_-{\res_{x}^\omega} \ar[rd]^-{\ev_{y}}& \\
 \kA\big|_{x} \ar[rr]^-{\rho^{\sharp}(x, y)} & & \kA\big|_{y}
}
\end{array}
\end{equation}
commutative, where $\mathsf{res}^\omega_x$ and $\mathsf{ev}_y$ are the canonical residue and evaluation maps; see
\cite[Definition 3.13]{BurbanGalinat}.
Then  the  tensor $\rho(x, y) \in \kA\big|_{x} \otimes \kA\big|_{y}$ (which is the value of the geometric $r$--matrix $\rho$  at the point $(x, y)$) is the image of the linear map $\rho^{\sharp}(x, y)$ under the isomorphism
$$\Hom_{\CC}\bigl(\kA\big|_{x}, \kA\big|_{y}\bigr) \lar \kA\big|_{x} \otimes \kA\big|_{y}$$ induced by the trace  form on the fiber $\kA\big|_{x}$.

 Next, it was explained in \cite[Subsection 5.1.4]{BK4} and \cite[Corollary 6.5]{BH} that a choice of homogeneous coordinates on $\PP^1$ together with  a choice of trivializations $\kO_{\PP^1}(k)\Big|_{Z} \stackrel{\xi_k}\lar \kO_{Z}$ specify   a trivialization $\Gamma(\breve{E}, \kA) \stackrel{\xi}\lar \lieg \otimes \Gamma\bigl(\breve{E}, \kO_E\bigr)$ and an embedding
$\Gamma\bigl(E, \kA(x)\bigr) \stackrel{\bar\xi}\lar \lieg[z]$
such that the following diagram of vector spaces
\begin{equation}
\begin{array}{c}
\xymatrix{
\kA\big|_{x} \ar[d]_-{\xi_x} & \Gamma\bigl(E, \kA(x)\bigr) \ar[r]^-{\mathsf{ev}_y} \ar[l]_-{\mathsf{res}^\omega_x} \ar[d]^-{\bar\xi} & \kA\big|_{y} \ar[d]^-{\xi_y} \\
\lieg  & \mathsf{Sol} \ar[r]^{\overline{\mathsf{ev}}_y} \ar[l]_{\overline{\mathsf{res}}_x} & \lieg
}
\end{array}
\end{equation}
is commutative, where
\begin{itemize}
\item The vector space $\mathsf{Sol} = \mathsf{Sol}\bigl((c,d), x\bigr):= \mathsf{Im}(\bar\xi) \subset \lieg[z]$ has the following description:
\begin{equation}\label{E:SolNode}
\mathsf{Sol} = \left\{\left.
F(z) = \left(
\begin{array}{c|c}
A_0 + z A_1 & B \\
\hline
C_0 + zC_1 + z^2 C_2 & D_0 + z D_1
\end{array}
\right)
 \; \right| \;
F_0 = -x F_\infty^\sharp
\right\},
\end{equation}
where
$
F_0 = \left(
\begin{array}{c|c}
A_0  & B \\
\hline
C_0  & D_0
\end{array}
\right)$ and $
F_\infty = \left(
\begin{array}{c|c}
A_1  & B \\
\hline
C_2  & D_1
\end{array}
\right).
$
\item For $F  \in \mathsf{Sol}$ we put:
$
\overline{\mathsf{ev}}_y(F) = \dfrac{1}{y-x} F(y)$ and
$
\overline{\mathsf{res}}_x(F) =
\dfrac{1}{x} F(x).
$
\end{itemize}
Let $r = \rho^\xi: \CC^\ast \times \CC^\ast \lar \lieg \otimes \lieg$ be the solution of the classical Yang--Baxter equation
(\ref{E:CYBE}),  given by  the geometric $r$--matrix $\rho$ trivialized by $\xi$. Then we get the following recipe to compute  $r(x,y)$  for $x \ne y \in \CC^*$: it is the image of the linear map
\begin{equation}\label{E:rsharp}
r_{x,y}^\sharp:= \overline{\ev_y} \circ \overline{\res}_x^{-1} \in \Hom_{\CC}(\lieg, \lieg)
\end{equation}  under the canonical isomorphism
$\Hom_{\CC}(\lieg, \lieg) \lar \lieg \otimes \lieg$,
induced by the trace form.

\smallskip
\noindent
Recall that we have two mutually inverse automorphisms $\tau, \kappa \in \Aut_\CC(\lieg)$ given by the formulae $\tau:= \mathrm{Ad}_{K_{(c,d)}}$ and
$\kappa:= \mathrm{Ad}_{J_{(c,d)}}$, respectively.

\begin{proposition}\label{P:keydecomp} For any $X \in \lieg$, there exists a uniquely determined
$$
Y = \bigl(Y', Y''\bigr) = \left(\left(
\begin{array}{c|c}
A & 0 \\
\hline
Z & D
\end{array}
\right),
\left(
\begin{array}{c|c}
A & 0 \\
\hline
C & D
\end{array}
\right)
\right) \in \nabla_{(c, d)}
$$
such that
$
X = Y' - \kappa(Y'') =
\left(
\begin{array}{c|c}
A & 0 \\
\hline
Z & D
\end{array}
\right) -
\left(
\begin{array}{c||c}
D & C \\
\hline
\hline
0 & A
\end{array}
\right).
$
Moreover, if $X \in \lieg_\pm$ (respectively, $X \in \lieh$) then we have: $Y', \kappa(Y'') \in \lieg_\pm$ (respectively, $Y', \kappa(Y'') \in \lieh$).
\end{proposition}

\begin{proof}
Existence and uniqueness of such $Y \in \nabla_{(c, d)}$ follows from the direct sum decomposition $\lieg \times \lieg = \Delta_{(c, d)} \dotplus
\nabla_{(c, d)}$; see Proposition \ref{P:maindecomp}. However, one can be more precise. Let
$
\bigl\{g_\beta\bigr\}_{\beta \in \Phi} = \bigl\{e_\alpha\bigr\}_{\alpha \in \Phi_+} \cup \bigl\{q_i\bigr\}_{i \in \Phi_0} \cup
\bigl\{e_{-\alpha}\bigr\}_{\alpha \in \Phi_+}
$
be the same basis as in Theorem \ref{T:maincomputation}. Then the following formulae are true:
$$
q_i = \tau(w_i) - w_i = \tau(w_i) - \kappa\bigl(\tau(w_i)\bigr) \quad \mbox{\rm for any}\quad 1 \le i \le n-1,
$$
$$
	e_\alpha = -\kappa \Bigl(\sum\limits_{l= 1}^{p(\alpha)}- e_{\tau^l(\alpha)}\Bigr) +
 \Bigl(\sum\limits_{l= 1}^{p(\alpha)-1}-e_{\tau^l(\alpha)}\Bigr) \quad \mbox{\rm for any}\quad \alpha \in \Phi_+
$$
and
$$
e_\beta = \Bigl(\sum\limits_{l= 0}^{t(\beta)-1} e_{\kappa^l(\beta)}\Bigr) + e_{\kappa^{t(\beta)}(\beta)} -
\kappa\Bigl(\sum\limits_{l= 0}^{t(\beta)-1}e_{\kappa^l(\beta)}\Bigr) \quad \mbox{\rm for any}\quad \beta \in \Phi_-.
$$
These formulae imply the result.
\end{proof}
\begin{proposition}\label{P:preimagesresidue}
Let $X =
\left(
\begin{array}{c|c}
M & N \\
\hline
K & L
\end{array}
\right) \in \lieg
$ be any element and
$$
X = \left(
\begin{array}{c|c}
A & 0 \\
\hline
Z & D
\end{array}
\right) -
\left(
\begin{array}{c||c}
D & C \\
\hline
\hline
0 & A
\end{array}
\right) \quad \mbox{\rm and}\quad
\left(
\begin{array}{c|c}
0 & N \\
\hline
0 & 0
\end{array}
\right)
= -\left(
\begin{array}{c|c}
A' & 0 \\
\hline
0 & D'
\end{array}
\right) +
\left(
\begin{array}{c||c}
D' & C' \\
\hline
\hline
0 & A'
\end{array}
\right)
$$
be the corresponding decompositions from Proposition \ref{P:keydecomp}. Let
$$
F(z) :=
-x
\left(
\begin{array}{c||c}
D + x D' & C \\
\hline
\hline
xN & A + x A'
\end{array}
\right)+
z \left(
\begin{array}{c|c}
A + xA' & 0 \\
\hline
Z + x(C'-C) & D + xD'
\end{array}
\right) +
z^2
\left(
\begin{array}{c|c}
0 & 0 \\
\hline
C & 0
\end{array}
\right).
$$
Then we have: $\overline{\res}_x^{-1}(X) = F(z)$.
\end{proposition}
\begin{proof}
We have to show that $\overline{\res}_x\bigl(F(z)\bigr) = X$ and $F(z) \in \Sol\bigl((c, d), x\bigr)$. Indeed,
$$
\overline{\res}_x\bigl(F(z)\bigr) = \frac{F(x)}{x} = \left(
\begin{array}{c|c}
A & 0 \\
\hline
Z & D
\end{array}
\right) -
\left(
\begin{array}{c||c}
D & C \\
\hline
\hline
0 & A
\end{array}
\right) + x\left(\left(
\begin{array}{c|c}
A' & 0 \\
\hline
C' & D'
\end{array}
\right) -
\left(
\begin{array}{c||c}
D' & 0 \\
\hline
\hline
N & A'
\end{array}
\right)\right) = X.
$$
Next, $F_0 = - x \left(
\begin{array}{c||c}
D + xD' & C \\
\hline
\hline
xN & A + xA'
\end{array}
\right)$. To compute $F_\infty$, first note that
$$
F_0 =
- x^2 \left(
\begin{array}{c||c}
D' & 0\\
\hline
\hline
N & A'
\end{array}
\right)  - x
\left(
\begin{array}{c||c}
D & C\\
\hline
\hline
0 & A
\end{array}
\right) = -x^2 \left(
\begin{array}{c|c}
A' & 0 \\
\hline
C' & D'
\end{array}
\right) - x
\left(
\begin{array}{c|c}
A & 0 \\
\hline
C & D
\end{array}
\right)
+ x
\left(
\begin{array}{c|c}
M & N \\
\hline
K & L
\end{array}
\right).
$$
Therefore, $F_\infty = \left(
\begin{array}{c|c}
A + xA'& xN \\
\hline
C & D + xD'
\end{array}
\right)$. We see that $F_0 = - xF_\infty^\sharp$, hence $F(z) \in \Sol\bigl((c, d), x\bigr)$, as asserted.
\end{proof}
\smallskip
\noindent
The following theorem is the main result of this section.
\begin{theorem}\label{T:GeomRmatrExplicit}
Let $r := \rho^{\xi}: \CC^+ \times \CC^\ast \lar \lieg \otimes \lieg$ be the trivialization of the geometric $r$--matrix, attached to the pair
$\bigl(E, \kA_{(n,d)}\bigr)$ (where $\xi$ is the trivialization of the sheaf of Lie algebras $\kA_{(n, d)}$ from the beginning of this section). Then $r(x,y)$ is given by the formula
(\ref{E:finalformula1}). In other words, $r(x,y)$ is the quasi--trigonometric solution attached to the order $\lieW_{(c, d)}$ given by formula
(\ref{E:CremmerGervaisOrder}).
\end{theorem}
\begin{proof} Let $
\bigl\{g_\beta\bigr\}_{\beta \in \Phi}$ be the same basis of the Lie algebra $\lieg$ as in Proposition \ref{P:keydecomp} and Theorem \ref{T:maincomputation}. We put: $F_{\beta,x}(z) := \overline{\res}_x^{-1}(g_\beta) \in \Sol\bigl((c, d), x\bigr)$. Then we have:  $r_{x, y}^\sharp(g_\beta) = \dfrac{1}{y-x}F_{\beta, x}(y)$, where $r_{x,y} \in \Hom_{\CC}(\lieg, \lieg)$ is the linear map defined by (\ref{E:rsharp}).  Therefore, we have the following formula:
$
r(x,y) = \dfrac{1}{y-x} \sum\limits_{\beta \in \Phi} g_\beta^\ast \otimes F_{\beta, x}(y).
$
Taking into account the explicit formulae for the elements $F_{\beta, x}(z)$, following from Proposition \ref{P:keydecomp} and Proposition \ref{P:preimagesresidue}, we arrive at the following formulae:
\begin{align*}
	\overline{\res}_x^{-1}(e_\alpha) &= x\kappa \Bigl(\sum\limits_{l= 1}^{p(\alpha)} e_{\tau^l(\alpha)}\Bigr) +z\Bigl(\sum\limits_{l= 1}^{p(\alpha)-1}-e_{\tau^l(\alpha)}\Bigr) +xz e_{\tau^{p(\alpha)}(\alpha)} - z^2 e_{\tau^{p(\alpha)}(\alpha)}\\
	r^\sharp_{x,y}(e_\alpha) &= \frac{x}{y-x} e_\alpha - \Bigl(\sum\limits_{l= 1}^{p(\alpha)-1}e_{\tau^l(\alpha)}\Bigr) - y e_{\tau^{p(\alpha)}(\alpha)}\\
	&= \frac{1}{2}\frac{y+x}{y-x}e_\alpha - \frac{1}{2}e_\alpha - \Bigl(\sum\limits_{l= 1}^{p(\alpha)-1}e_{\tau^l(\alpha)}\Bigr) - y e_{\tau^{p(\alpha)}(\alpha)} \quad \mbox{\rm for any}\; \alpha \in \Phi_+, \; \kappa(\alpha) \notin \Phi_- \\
	\overline{\res}_x^{-1}(e_\alpha) &=  x\kappa \Bigl(\sum\limits_{l= 1}^{p(\alpha)} e_{\tau^l(\alpha)}\Bigr)-x^2\Bigl(\sum\limits_{l= 1}^{t(\kappa(\alpha))} e_{\kappa^l(\alpha)}\Bigr) + xz e_{\tau^{p(\alpha)}(\alpha)} +z\Bigl(\sum\limits_{l= 1}^{p(\alpha)-1}-e_{\tau^l(\alpha)}\Bigr) \\&+xz \Bigl(\sum\limits_{l= 1}^{t(\kappa(\alpha))} e_{\kappa^l(\alpha)}\Bigr) -z^2 e_{\tau^{p(\alpha)}(\alpha)}\\
r^\sharp_{x,y}(e_\alpha) &= \frac{x}{y-x} e_\alpha - \Bigl(\sum\limits_{l= 1}^{p(\alpha)-1}e_{\tau^l(\alpha)}\Bigr) -y e_{\tau^{p(\alpha)}(\alpha)} +x \Bigl(\sum\limits_{l= 1}^{t(\kappa(\alpha))} e_{\kappa^l(\alpha)}\Bigr)\\
	&= \frac{1}{2}\frac{y+x}{y-x}e_\alpha + \frac{1}{2}e_\alpha +  \Bigl(\sum\limits_{l= 1}^{p(\alpha)-1}e_{\tau^l(\alpha)}\Bigr) -y e_{\tau^{p(\alpha)}(\alpha)} +x \Bigl(\sum\limits_{l= 1}^{t(\kappa(\alpha))} e_{\kappa^l(\alpha)}\Bigr) \; \mbox{\rm for any}\; \alpha \in \Phi_+, \; \kappa(\alpha) \in \Phi_-. \\
\end{align*}
Similarly, we have:
\begin{align*}
\overline{\res}_x^{-1}(e_\beta) &=  z\biggl(\Bigl(\sum\limits_{l= 0}^{t(\beta)-1} e_{\kappa^l(\beta)}\Bigr) + e_{\kappa^{t(\beta)}(\beta)}\biggr) -x
\kappa\Bigl(\sum\limits_{l= 0}^{t(\beta)-1}e_{\kappa^l(\beta)}\Bigr) \\
	r^\sharp_{x,y}(e_\beta) &= \frac{y}{y-x} e_\beta + \Bigl(\sum\limits_{l= 1}^{t(\beta)}e_{\kappa^l(\beta)}\Bigr) \\
	&= \frac{1}{2}\frac{y+x}{y-x}e_\beta + \frac{1}{2}e_\beta + \Bigl(\sum\limits_{l= 1}^{t(\beta)}e_{\kappa^l(\beta)}\Bigr) \quad \mbox{\rm for any}\quad \beta \in \Phi_- \\
\overline{\res}_x^{-1}(q_i) &= z \tau(w_i) - x w_i \\
r^\sharp_{x,y}(q_i) &= \frac{x}{y-x}q_i + \tau(w_i) = \frac{1}{2}\frac{y+x}{y-x} q_i + f_i;
\end{align*}
see (\ref{E:basish}) for the definition of the elements $q_i$, $w_i$ and $f_i$ of the Lie algebra $\lieh$.
Since $\kappa^{-l}(i,j) = (a,b)$ if and only if $\tau^l(b,a) = (j,i)$ the claim now follows by comparison with (\ref{E:finalformula1}).
\end{proof}

\begin{remark}
  In \cite{Polishchuk2}, Polishchuk derived  explicit formulae for the solutions of the associative Yang--Baxter equation (AYBE)
  for $\mathsf{Mat}_{n \times n}(\CC)$, arising from simple vector bundles  on an arbitrary Kodaira cycle $E$ of projective lines. However,
  his answer used a  different combinatorial pattern, based on the recursive description  of discrete parameters describing the multi--degrees of  simple vector bundles on $E$, obtained in  the works \cite{OldSurvey,Burban1}. It follows from the analysis made by  Schedler \cite{Schedler} that not every
  quasi--constant (quasi--)trigonometric solutions of CYBE for the Lie algebra $\mathfrak{sl}_n(\CC)$ can be lifted to a solution
  of the associative Yang--Baxter equation for $\mathsf{Mat}_{n \times n}(\CC)$. Therefore, the combinatorial patterns   of the
  trigonometric solutions  of CYBE for $\mathfrak{sl}_n(\CC)$ (see \cite{BelavinDrinfeld}), of the quasi--trigonometric solutions  of the
  CYBE for $\mathfrak{sl}_n(\CC)$ (see \cite{KPSST, PopStolin}) and of the trigonometric solutions of the
  AYBE  for  $\mathsf{Mat}_{n \times n}(\CC)$ (see \cite{Polishchuk2, LekiliPolishchuk}) share similar features  but are different.

  In the light of the work \cite{BurbanGalinat}, it is natural to expect that any trigonometric solution of the
  CYBE arises as the  geometric $r$--matrix of  a pair $(E, \kA)$, where $E$ is a nodal
  Weierstra\ss{} cubic (such realizability is known for all elliptic and rational solutions; see \cite{BurbanGalinat}
  and references therein). On the other hand, the appearance of the Fukaya categories of
  higher genus Riemann surfaces in the classification of trigonometric solutions of AYBE \cite{LekiliPolishchuk} indicates, that a  geometrization of the trigonometric solutions of CYBE could lead to further surprises.

  It would be  interesting to find out, which  trigonometric solutions of CYBE  are gauge equivalent to quasi--trigonometric ones. It is also quite unclear, how the Belavin--Drinfeld combinatorics of the
  the trigonometric solutions of CYBE \cite{BelavinDrinfeld} is reflected in the geometric properties of the sheaf of Lie algebras $\kA$.  A further study of these relations will be a subject of future work.
\end{remark}

\begin{corollary}\label{C:Limits} First observe  that the map $\mathfrak{sl}_n(\CC) \lar \mathfrak{sl}_n(\CC), X \mapsto - X^t$ is a Lie algebra  automorphism, where $X^t$ denotes the transposed matrix of $X$.
  Sheafifying this map, we get an isomorphism of the sheaves of Lie algebras $\kA_{(n, d)} \lar \kA_{(n, -d)}$. Hence, we get isomorphisms $\kA_{(n, d)} \cong  \kA_{(n, -d)} \cong \kA_{(n, n-d)} = \kA_{(n, c)}$.
  As a consequence of the geometric theory of CYBE,
we get the following analytic results about the quasi--trigonometric $r$--matrix $r_{(c,d)}$ given by the formula (\ref{E:finalformula1}).
\begin{enumerate}
\item  The solutions $r_{(c,d)}$ and $r_{(d, c)}$ are gauge equivalent.
\item Moreover, $r_{(c,d)}$   is equivalent to an appropriate  degeneration of Belavin's elliptic $r$--matrix \cite{Belavin} corresponding to the primitive $n$--root of unity $\varepsilon = \exp\left(\dfrac{2\pi i c}{n}\right)$.
\item Finally, its appropriate rational degeneration is equivalent  to the rational solutions from \cite[Theorem 9.8]{BH}.
\end{enumerate}
\end{corollary}
\section{Dedication}
This paper is dedicated to the memory of Petr Kulish, one of the leaders and driving forces of the Yang--Baxter revolution
which has changed the world of mathematics and mathematical physics. He was always interested in new solutions of the
Yang--Baxter equation. The following list of problems may serve as our epitaph to him.

\smallskip
\noindent
1. Let $\mathfrak{g}$ be any finite dimensional simple Lie algebra over the field of complex numbers.
Let $\widehat{\Gamma}$ be its extended Dynkin diagram. Next,
\begin{itemize}
\item let $\Gamma_1$ and $\Gamma_2$ be two subdiagrams of $\widehat{\Gamma}$ and $\tau: \Gamma_1 \to \Gamma_2$ be a map such
  that
\begin{itemize}
\item $(\alpha , \beta) = \bigl(\tau(\alpha), \tau (\beta)\bigr)$ for any $\alpha ,\beta \in \Gamma_1$;
\item for any $\alpha \in \Gamma_1$ there exists $m \in \NN$ such that $\tau^m(\alpha) \notin \Gamma_1$;
\item $\Gamma_1$ does not contain the affine simple root $\alpha_0\in \widehat{\Gamma}$.
\end{itemize}
\item Let $r_\circ \in \lieh \wedge \lieh$ be  such that for any $\alpha\in\Gamma_1 $ we have:
$$
  \bigl(\alpha \otimes 1 + 1\otimes \tau(\alpha)\bigr)\left(r_\circ + \frac{\gamma_\circ}{2}\right) = 0,
  $$
where $\gamma_\circ$ is the Cartan part of the Casimir element $\gamma\in {\mathfrak{g}^{\otimes 2}}$.
\end{itemize}
Prove that
the following formula defines a quasi--trigonometric solution of CYBE and any quasi--trigonometric solution of CYBE
is gauge equivalent to some solution of this form:
\begin{equation}\label{E:Conjecture}
r(x,y)=r_{\mathsf{st}} (x,y) +  r_\circ + \sum_{\alpha\in \mathsf{Span}\{\Gamma_1^+ \}}\sum_{k} e_{\tau^k (\alpha)}\wedge e_{-\alpha}.
\end{equation}
Here, if $\beta = \tau^k(\alpha)$ is an affine root (i.e. if it contains $\alpha_0$), then $e_\beta$ reads as $e_\beta x$ if it stands
on the left hand side of the tensor product and as $e_\beta y$ if it stands on the right hand side.

\smallskip
\noindent
2. The same datum  $\bigl(\Gamma_1, \Gamma_2 , \tau\bigr)$ defines a trigonometric solution of CYBE; see
\cite{BelavinDrinfeld}.
What is a connection between these two solutions, trigonometric and quasi--trigonometric?

\smallskip
\noindent
3. How to quantize the quasi--trigonometric solution above? Is it possible to ``affinize'' quantum twists which were constructed
in \cite{ESS} as this was done in \cite{KPSST} in some particular cases?

\end{document}